\documentclass[11pt]{article}
\usepackage{amsfonts}
\usepackage{amssymb}
\usepackage{amsmath}
\usepackage{graphicx}

\usepackage{citesort}

\newtheorem{definition}{Definition}
\newtheorem{remark}{Remark}
\newtheorem{theorem}{Theorem}
\newtheorem{lemma}{Lemma}

\newtheorem{corollary}{Corollary}
\newenvironment{proof}{{\bf Proof.}}{$ \Box $}

\title{A generalized Calder\'on formula for open-arc diffraction problems: theoretical considerations}

\author{St\'ephane K. Lintner and Oscar P. Bruno}
\begin{document}

\maketitle
\begin{abstract}
  We deal with the general problem of scattering by open-arcs in
  two-dimensional space. We show that this problem can be solved by
  means of certain second-kind integral equations of the form
  $\tilde{N} \tilde{S}[\varphi] = f$, where $\tilde{N}$ and
  $\tilde{S}$ are first-kind integral operators whose composition
  gives rise to a generalized Calder\'on formula of the form
  $\tilde{N} \tilde{S} = \tilde{J}_0^\tau + \tilde{K}$ in a {\em
    weighted, periodized} Sobolev space. (Here $\tilde{J}^\tau_0$ is a
  continuous and continuously invertible operator and $\tilde{K}$ is a
  compact operator.) The $\tilde{N} \tilde{S}$ formulation provides,
  for the first time, a second-kind integral equation for the open-arc
  scattering problem with Neumann boundary conditions. Numerical
  experiments show that, for both the Dirichlet and Neumann boundary
  conditions, our second-kind integral equations have spectra that are
  bounded away from zero and infinity as $k\to \infty$; to the
  authors' knowledge these are the first integral equations for these
  problems that possess this desirable property. This situation is in
  stark contrast with that arising from the related {\em classical}
  open-surface hypersingular and single-layer operators $\mathbf{N}$
  and $\mathbf{S}$, whose composition $\mathbf{NS}$ maps, for example,
  the function $\phi =1$ into a function that is not even square
  integrable. Our proofs rely on three main elements: 1)~Algebraic
  manipulations enabled by the presence of integral weights; 2)~Use of
  the classical result of continuity of the Ces\`aro operator; and
  3)~Explicit characterization of the point spectrum of
  $\tilde{J}^\tau_0$, which, interestingly, can be decomposed into the
  union of a countable set and an open set, both of which are tightly
  clustered around $-\frac{1}{4}$.  As shown in a separate
  contribution, the new approach can be used to construct simple
  spectrally-accurate numerical solvers and, when used in conjunction
  with Krylov-subspace iterative solvers such as GMRES, it gives rise
  to dramatic reductions of iteration numbers vs. those required by
  other approaches.
\end{abstract}



\section{Introduction\label{intro}}

The field of Partial Differential Equations (PDEs) with boundary
values prescribed on open surfaces has a long and important history,
including significant contributions in the theory of diffraction by
open screens, elasticity problems in solids containing cracks, and
fluid flow past plates; solution to such problems impact significantly
on present day technologies such as wireless transmission, electronics
and photonics. From a mathematical point of view, besides techniques
applicable to simple geometries, existing solution methods include
special adaptations of finite-element and boundary-integral methods
that account in some fashion for the singular character of the PDE
solutions at edges. With much progress in the area over the last sixty
years the field remains challenging: typically only low-frequency
open-surface problems can be treated with any accuracy by previous
approaches. 

In this paper we focus on the problem of electromagnetic and acoustic
scattering by open arcs. In particular, we introduce certain
first-kind integral operators $\mathbf{N}_\omega$ and
$\mathbf{S}_\omega$ whose composition gives rise, after appropriate
change of periodic variables, to a generalized Calder\'on formula
$\tilde{N}\tilde{S} = \tilde{J}^\tau_0 + \tilde{K}$---where
$\tilde{J}^\tau_0$ is a continuous and continuously invertible
operator and where $\tilde{K}$ is a compact operator---together with
associated {\em second-kind open-surface integral equations} of the
form $\tilde{N} \tilde{S}[\varphi] = f$.  This approach enables, for
the first time, treatment of open-arc scattering problems with Neumann
boundary conditions by means of second kind equations. Further, a wide
range of numerical experiments~\cite{BrunoLintner2} indicate that, for
both the Dirichlet and Neumann boundary conditions, our second-kind
integral equations have spectra that are bounded away from zero and
infinity as $k\to \infty$, and give rise to high accuracies and
dramatic reductions of Krylov-subspace iteration numbers vs. those
required by other approaches. These methods and results were first
announced in~\cite{BrunoOberwolfach}; succinct proofs of the open-arc
Calder\'on formulae, further, were presented in~\cite{BrunoLintner2}.

Integral equation methods provide manifold advantages over other
methodologies: they do not suffer from the well known
pollution~\cite{BabuskaSauterStefan} and dispersion~\cite{Jameson}
errors characteristic of finite element and finite difference methods,
they automatically enforce the condition of radiation at infinity
(without use of absorbing boundary conditions), and they lend
themselves to (typically iterative) acceleration
techniques~\cite{BleszynskiBleszynskiJaroszewicz,BrunoKunyansky,Rokhlin}---which
can effectively take advantage of the reduced dimensionality arising
from boundary-integral equations, even for problems involving very
high-frequencies. Special difficulties inherent in open-surface
boundary-integral formulations arise from the solution's edge
singularity~\cite{Maue,Stephan,Costabel}. Such difficulties have
typically been tackled by incorporating the singularity explicitly in
both Galerkin~\cite{Stephan,StephanWendland,StephanWendland2} and
Nystr\"om~\cite{AtkinsonSloan,RokhlinJiang,Monch} integral solvers;
with one exception (introduced in the
contributions~\cite{AtkinsonSloan,RokhlinJiang} and discussed below in
some detail), in all of these cases integral equations of the first
kind were used. While providing adequate discretizations of the
problem, first-kind integral equation can be poorly conditioned and,
for high-frequencies, they require large numbers of iterations and
long computing times when accelerated iterative solvers as mentioned
above are used.

(The literature on the singular behavior of open-arc solutions is
quite rich and interesting from a historical perspective: it includes
the early analysis~\cite{Sommerfeld},
corrections~\cite{BouwkampReview,BouwkampOnBethe} to early
contributions~\cite{MeixnerOld,Bethe}, the well-known finite-energy
condition introduced in~\cite{Meixner}, the integral equation
formulation~\cite{Maue} and subsequent treatments for integral
approaches for these problems, leading to the first regularity
proof~\cite{Stephan} and the comprehensive treatment~\cite{Costabel}
which establishes, in particular, that for $C^\infty$ open surfaces
with $C^\infty$ edges, the integral equation solution of the Dirichlet
(resp. Neumann) open-edge problems equals a $C^\infty$ function times
an unbounded (resp. bounded) canonical edge-singular function.)

As mentioned above, iterative solvers based on first-kind integral equations
often require large numbers of iterations and long computing times. Attempts
have been made over the years to obtain second-kind open-surface equations
and, indeed, second kind equations for open surfaces were developed
previously by exploiting the diagonal character of the logarithmic single
layer, at least for the case of the {\em Dirichlet problem for the Laplace
equation}~\cite{AtkinsonSloan,RokhlinJiang}. Unfortunately, as shown
in~\cite{BrunoLintner2}, direct generalization of such approaches to
high-frequency problems give rise to numbers of iterations that can in fact
be much larger than those inherent in first-kind formulations.  Efforts were
also made to obtain second kind equations on the basis of the well known
Calder\'on formula.  The Calder\'on identity relates the classical
single-layer and hypersingular operators $\mathbf{S}_c$ and $\mathbf{N}_c$
that are typically associated with the Dirichlet and Neumann problems on a
closed surface $\Gamma_c$: for such closed surfaces the Calder\'on formula
reads $\mathbf{N}_c\mathbf{S}_c = -\mathbf{I}/4 + \mathbf{K}_c$, where
$\mathbf{K}_c$ is a compact operator in a suitable Sobolev space.  Attempts
to extend this idea to open surfaces were pursued
in~\cite{PovznerSuharesvki,ChristiansenNedelec}. As first shown
in~\cite{PovznerSuharesvki}, there indeed exists a related identity for the
corresponding single-layer and hypersingular operators $\mathbf{N}$ and
$\mathbf{S}$ on an open surface $\Gamma$. As in the closed-surface case, we
have $\mathbf{N}\mathbf{S} = -\mathbf{I}/4 + \mathbf{K}$; unfortunately,
however, a useful functional setting for the operator $\mathbf{N}\mathbf{S}$
does not appear to exist ($\mathbf{K}$ is not compact). As shown in
Appendix~\ref{NSOne}, for example, the composition $\mathbf{N}\mathbf{S}$
maps the constant function $\varphi =1$ on an open surface into a function
that tends to infinity at the boundary of $\Gamma$ like $1/d$, where $d$
denotes the distance to the curve edge (see Appendix~\ref{NSOne}). In
particular, the formulation $\mathbf{N}\mathbf{S}$ cannot be placed in the
functional framework put forth
in~\cite{Stephan,StephanWendland,StephanWendland2} and embodied by
equations~\eqref{Smapping},~\eqref{Nmapping} and Definition~\ref{htil_def}
below: a function with $1/d$ edge asymptotics is not an element of
$H^{-\frac{1}{2}}(\Gamma)$.

In view of the aforementioned regularity
result~\cite{Costabel}---which, can be fully exploited numerically
through use of Chebyshev expansions~\cite{AtkinsonSloan,Monch} in
two-dimensions and appropriate extensions~\cite{BrunoLintner3} of the
high-order integration methods~\cite{BrunoKunyansky} in
three-dimensions---, and on account of the results of this paper, use
of the combination $N_\omega S_\omega$ enables low-iteration-number,
second-kind, super-algebraically accurate solution of open surface
scattering problems---and thus gives rise to a highly efficient
numerical solver for open surface scattering problems in two and three
dimensions; see~\cite{BrunoLintner2,BrunoLintner3}.

This paper is organized as follows. After introduction of necessary
notations and preliminaries, the main result of this paper,
Theorem~\ref{th_1}, is stated in Section~\ref{sec_2}.  Section~\ref{three}
contains the main elements of the theory leading to a proof of
Theorem~\ref{th_1}. Necessary uniqueness and regularity results for the
single-layer and hyper-singular weighted operators under a cosine change of
variables, which have mostly been known for a number of years
(see~\cite[Ch. 11]{VainikkoSaranen} and the extensive literature cited
therein) are presented in Section~\ref{SGeneral}; the inclusion of these
results renders our text essentially self contained, and it establishes a
direct link between our context and that represented by
references~\cite{Stephan,StephanWendland,StephanWendland2}. Building up on
constructions presented in earlier sections, the proof of Theorem~\ref{th_1}
is given in Section~\ref{five}. Two appendices complete our contribution:
Appendix~\ref{Nfact_appendix} presents a version adequate to our context of
a known expression linking the hypersingular operator and an
integro-differential operator containing only tangential derivatives;
Appendix~\ref{NSOne}, finally, demonstrates that the image of the
composition $\mathbf{N}\mathbf{S}$ of the un-weighted operators $\mathbf{N}$
and $\mathbf{S}$ is not contained in $H^{-\frac{1}{2}}$.

\section{Preliminaries\label{sec_2}}
Throughout this paper $\Gamma$ is assumed to be a smooth open arc in
two dimensional space.
\subsection{Background}
As is well
known~\cite{StephanWendland,StephanWendland2,VainikkoSaranen}, the
Dirichlet and Neumann boundary-value problems for the Helmholtz
equation
\begin{equation}\label{b_conds}
\left\{ \begin{array}{llll} \Delta u +k^2 u = 0 \quad\mbox{outside}\quad
  \Gamma , & u|_{\Gamma} = f, &f \in H^\frac{1}{2}(\Gamma) & \quad\mbox{(Dirichlet)}\\
\Delta v +k^2 v = 0 \quad\mbox{outside}\quad \Gamma, & \frac{\partial
  v}{\partial n}|_{\Gamma} = g,  & g\in H^{-\frac{1}{2}}(\Gamma) &\quad\mbox{(Neumann)}
\end{array}\right.
\end{equation}
admit unique radiating solutions $u,v\in
H^1_\text{loc}(\mathbb{R}^2\setminus \Gamma)$ which can be expressed
in terms of single- and double-layer potentials, respectively:
 \begin{equation}\label{single_l}
 u(\mathbf{r})= \int_\Gamma G_k(\mathbf{r},\mathbf{r}')\mu(\mathbf{r}')
 d\ell'
 \end{equation}
and
 \begin{equation}
 v(\mathbf{r})= \int_\Gamma \frac{\partial
 G_k(\mathbf{r},\mathbf{r}')}{\partial
 \textbf{n}_{\mathbf{r}'}}\nu(\mathbf{r}') d\ell' 
 \end{equation}
 for $\mathbf{r}$ outside $\Gamma$. Here $\textbf{n}_{\mathbf{r}'}$ is
 a unit vector normal to $\Gamma$ at the point $\mathbf{r}'\in\Gamma$
 (we assume, as we may, that $\textbf{n}_{\mathbf{r}'}$ is a smooth
 function of $\mathbf{r}'\in\Gamma$), and, letting $H_0^1$ denote the
 Hankel function, 
\begin{equation}\label{Gk_def}
G_k(\mathbf{r},\mathbf{r}')=\left\{\begin{array}{ll} \frac{i}{4}H_0^1(k
|\mathbf{r}-\mathbf{r}'|),& k > 0 \\
-\frac{1}{2\pi}\ln|\mathbf{r}-\mathbf{r}'|,&k=0 \end{array}\right. ,
\end{equation}
and
\begin{equation}
\frac{\partial G_k(\mathbf{r},\mathbf{r}')}{\partial
  \mathbf{n}_{\mathbf{r}'}}=\mathbf{n}_{\mathbf{r}'}\cdot\nabla_{\mathbf{r}'}G_k(\mathbf{r},\mathbf{r}').
\end{equation}

Denoting by $\mathbf{S}$ and $\mathbf{N}$ the single-layer and
hypersingular operators
\begin{equation}\label{Sdef}
 \mathbf{S}[\mu](\mathbf{r})=  \int_\Gamma
 G_k(\mathbf{r},\mathbf{r}')\mu(\mathbf{r}') d\ell'\quad , \quad
 \mathbf{r}\in\Gamma,
 \end{equation}
and
\begin{equation}\label{Ndef}
\begin{split}
  \mathbf{N}[\nu](\mathbf{r})= &\; \frac{\partial }{\partial
  \textbf{n}_{\mathbf{r}}}\int_\Gamma \frac{\partial
  G_k(\mathbf{r},\mathbf{r}')}{\partial
  \textbf{n}_{\mathbf{r}'}}\nu(\mathbf{r}') d\ell'\\
  \stackrel{\mathrm{def}}{=}&\lim\limits_{z\rightarrow 0^+} \frac{\partial
  }{\partial z}\int_\Gamma \frac{\partial
  G_k(\mathbf{r}+z\mathbf{n}_{\mathbf{r}},\mathbf{r}')}{\partial
  \textbf{n}_{\mathbf{r}'}}\nu(\mathbf{r}') d\ell'\quad , \quad
  \mathbf{r}\in\Gamma,
\end{split}
\end{equation}
the densities $\mu$ and $\nu$ are the unique solutions of the first kind
integral equations
\begin{equation}\label{Sbad}
\mathbf{S}[\mu]=f
\end{equation}
and
\begin{equation}\label{Nbad}
\mathbf{N}[\nu]=g.
\end{equation}
As shown in~\cite{Stephan,StephanWendland,StephanWendland2}, the operators
$\mathbf{S}$ and $\mathbf{N}$ define bounded and continuously invertible
mappings
\begin{equation}\label{Smapping}
  \mathbf{S}:\; \tilde{H}^{-\frac{1}{2}}(\Gamma) \rightarrow H^{\frac{1}{2}}(\Gamma),\quad\mbox{and}
\end{equation}
\begin{equation}\label{Nmapping}
  \mathbf{N}:\; \tilde{H}^{\frac{1}{2}}(\Gamma) \rightarrow H^{-\frac{1}{2}}(\Gamma),
\end{equation}
where for $s\in \mathbb{R}$, the space $\tilde{H}^s(\Gamma)$ is
defined below.
\begin{definition}\label{htil_def} 
  Let $G_1$ be a domain in the plane, with a smooth boundary $\dot G_1$, let
  $s\in\mathbb{R}$, and assume $\dot G_1$ contains the smooth open curve
  $\Gamma$.  The Sobolev space $\tilde{H}^{s}(\Gamma)$ is defined as the set
  of all elements $f\in H^{s}(\dot G_1)$ satisfying
  $\mathrm{supp}(f)\subseteq\overline\Gamma$.
\end{definition}
\begin{remark}\label{inv_is_cont} 
  As is well known~\cite[Corollary 2.7]{brezis2010functional}, the
  inverse $L^{-1}$ of a continuous and invertible (one-to-one and
  surjective) operator $L$ between two Banach spaces (and, in
  particular, between two Hilbert spaces such as the Sobolev spaces
  considered in this text) is also continuous. In view of this fact,
  above and throughout this text the terms ``invertible continuous
  operator'', ``invertible bounded operator'', ``bicontinuous
  operator'', ``continuous operator with continuous inverse'', etc,
  are used as interchangeable synonyms.
\end{remark}

The mapping results~\eqref{Smapping},~\eqref{Nmapping} provide an
extension of classical closed-surfaces results: for a closed Lipschitz
surface $\Gamma_c$ and for any $s\in\mathbb{R}$, the closed-surface
single-layer and hypersingular operators $\mathbf{S}_c$ and
$\mathbf{N}_c$ define bounded mappings
\begin{equation}\label{ScmappingHs}
  \mathbf{S}_c:\; H^{s}(\Gamma_c) \rightarrow H^{s+1}(\Gamma_c), 
\end{equation}
\begin{equation}\label{NcmappingHs}
  \mathbf{N}_c:\; H^{s+1}(\Gamma_c) \rightarrow H^{s}(\Gamma_c),
\end{equation}
see e.g.~\cite{Costabel_Old,Nedelec,Kress}. Additionally, the
closed-surface potentials satisfy the classical Calder\'on relation
\begin{equation}\label{Calderon}
  \mathbf{N}_c\mathbf{S}_c= -\frac{\mathbf{I}}{4} +\mathbf{K}_c
\end{equation}
in $H^{s}(\Gamma_c)$, where $\mathbf{K}_c$ is a compact operator.

While~\eqref{ScmappingHs} and~\eqref{NcmappingHs} do not apply to open
surfaces, the solutions $\mu$ and $\nu$ of the open-surface integral
equations~\eqref{Sbad} and~\eqref{Nbad} enjoy significant regularity
properties.  In particular, letting $d=d(\mathbf{r})$ denote any
non-negative smooth function defined on $\Gamma$ which for
$\mathbf{r}$ in a neighborhood of each end-point equals the Euclidean
distance from $\mathbf{r}$ to the corresponding end-point, and letting
$\omega$ denote any function defined on $\Gamma$ such that
$\omega/\sqrt{d}$ is $C^\infty$ up to the endpoints, the recent
results~\cite{Costabel} establish that if the arc $\Gamma$ and the
right-hand-side functions $f$ and $g$ in~(\ref{Sbad}) and~(\ref{Nbad})
are infinitely differentiable we have
\begin{equation}\label{CostabelExp1}
  \mu = \frac{\alpha}{\omega}
\end{equation}
and
\begin{equation}\label{CostabelExp2}
  \nu = \beta\cdot\omega,
\end{equation}
where $\alpha$ and $\beta$ are $C^\infty$ functions throughout
$\Gamma$. The singular behavior in these solutions is thus fully
characterized by the factors $d^{1/2}$ and $d^{-1/2}$ in
equations~\eqref{CostabelExp1} and~\eqref{CostabelExp2}, respectively.

\subsection{Generalized Calder\'on Formula\label{well_cond}}

In view of equations~\eqref{CostabelExp1} and~\eqref{CostabelExp2},
for any non-vanishing function $\omega(\mathbf{r})>0$ such that $\omega/\sqrt{d}$ is
$C^\infty$ up to the endpoints of $\Gamma$, we define the weighted operators
\begin{equation}\label{Somega}
\mathbf{S}_\omega[\alpha]=  \mathbf{S}\left[\frac{\alpha}{\omega}\right]
\end{equation}
and
\begin{equation}\label{Nomega}
\mathbf{N}_\omega[\beta] = \mathbf{N}\left[\beta\cdot\omega\right],
\end{equation}
and we consider the weighted versions
\begin{equation}\label{Sgood}
\mathbf{S}_\omega[\alpha]= f
\end{equation}
and
\begin{equation}\label{Ngood}
\mathbf{N}_\omega[\beta] = g
\end{equation}
of the integral equations~\eqref{Sbad} and~\eqref{Nbad}; clearly, in
view of the discussion of the previous section, for smooth $\Gamma$
and smooth right-hand-sides $f$ and $g$, the solutions $\alpha$ and
$\beta$ of~\eqref{Sgood} and~\eqref{Ngood} are smooth up to the
endpoints of $\Gamma$.

Without loss of generality we use a smooth parametrization
$\mathbf{r}(t)=\left(x(t),y(t)\right)$ of $\Gamma$ defined in the interval
$[-1,1]$, for which $\tau(t)=|\frac{d \textbf{r}(t)}{dt}|$ is never zero.
For definiteness and simplicity, throughout the rest of the paper we select $\omega$, as we
may, in such a way that
\begin{equation}\label{canon_sq}
\omega(\mathbf{r}(t))= \sqrt{1-t^2}. 
\end{equation}
The operators $\mathbf{S}_\omega$ and $\mathbf{N}_\omega$ thus induce the
parameter-space operators
\begin{equation}\label{sOmegaDef}
S_\omega[\varphi](t)=\int_{-1}^1 G_k\left(\mathbf{r}(t),\mathbf{r}(t')\right)\frac{
\varphi(t')}{\sqrt{1-t'^2}}\;\tau(t') dt',
\end{equation}
and
\begin{equation}\label{NOmegaDef}
N_\omega[\psi](t)=\lim\limits_{z\rightarrow
0^+}\frac{\partial}{\partial z}\int_{-1}^1 
\frac{\partial}{\partial
\textbf{n}_{\mathbf{r}(t')}}G_k\left(\mathbf{r}(t)+z\textbf{n}_{\mathbf{r}(t)},\mathbf{r}(t')\right)
\psi (t') \tau(t')\sqrt{1-t'^2}dt';
\end{equation}
defined on functions $\varphi$ and $\psi$ of the variable $t$, $-1\leq
t\leq 1$; clearly, for $\varphi(t) = \alpha(\mathbf{r}(t))$ and $\psi(t)
= \beta(\mathbf{r}(t))$ we have
\begin{equation}\label{S_param}
\mathbf{S}_\omega[\alpha](\mathbf{r}(t)) = S_\omega[\varphi](t)
\end{equation}
and
\begin{equation}\label{N_param}
\mathbf{N}_\omega[\beta](\mathbf{r}(t)) = N_\omega[\psi](t).
\end{equation}

In order to proceed we further transform our integral operators: using
the changes of variables $t=\cos\theta$ and $t'=\cos\theta'$ and,
defining $\textbf{n}_\theta= \textbf{n}_{\mathbf{r}(\cos\theta)}$ and
using~\eqref{S_param} and~\eqref{N_param}, we re-express
equations~\eqref{Sgood} and~\eqref{Ngood} in the forms
\begin{equation}\label{Stilde}
\tilde{S}[\tilde{\varphi}]= \tilde f
\end{equation}
and
\begin{equation}\label{Ntilde}
\tilde{N}[\tilde{\psi}] = \tilde g,
\end{equation}
where $\tilde{S}$ and $\tilde{N}$ denote the operators
\begin{equation}\label{ssdef}
  \tilde{S}[\gamma](\theta)=\int_{0}^\pi
  G_k(\mathbf{r}(\cos\theta),\mathbf{r}(\cos\theta'))\gamma(\theta')\tau(
  \cos\theta')d\theta'
\end{equation}
and
\begin{equation}\label{NNdef}
  \tilde{N}[\gamma](\theta)=\lim\limits_{z\rightarrow
    0^+}\frac{\partial}{\partial z}\int_{0}^\pi
  \frac{\partial}{\partial
    \textbf{n}_{\theta'}}G_k(\mathbf{r}(\cos\theta)+z\textbf{n}_{\theta},\mathbf{r}(\cos\theta'))\gamma(\theta')
  \tau(\cos\theta')\sin^2\theta'd\theta',
\end{equation}
and where
\begin{equation}\label{fg_theta}
  \tilde f(\theta) = f(\mathbf{r}(\cos\theta))\quad,\quad \tilde g(\theta) = g(\mathbf{r}(\cos\theta));
\end{equation}
clearly, the solutions of equations~\eqref{sOmegaDef}-\eqref{Ntilde}
are related by
\begin{equation}\label{phi_psi_theta}
  \tilde \varphi(\theta) = \varphi(\cos\theta)\quad,\quad \tilde \psi(\theta)
  = \psi(\cos\theta).
\end{equation}

In view of the symmetries induced by the $\cos\theta$ dependence in
equations~\eqref{ssdef} through \eqref{fg_theta}, it is natural to study the
properties of these operators and equations in appropriate Sobolev spaces
$H^s_e(2\pi)$ of $2 \pi$ periodic and even functions defined below; cf.
\cite{YanSloan,BrunoHaslam}.
\begin{definition}
  Let $s\in \mathbb{R}$. The Sobolev space $H^s_e(2\pi)$ is defined as
  the completion of the space 
  of infinitely differentiable $2\pi$-periodic and even functions
  defined in the real line with respect to the norm
\begin{equation}
\|v\|_{H_s^e(2\pi)}^2 = |a_0|^2 + 2\sum\limits_{m=1}^\infty m^{2s}|a_m|^2 ,
\end{equation}
where $a_m$ denotes the $m$-th cosine coefficient of $v$:
\begin{equation}
v(\theta)=\frac{1}{2}a_0 + \sum\limits_{m=1}^\infty a_m \cos( m\theta ).
\end{equation}
\end{definition}
Clearly the set $\{\cos(n\theta):n\in \mathbb{N} \}$ is a basis of the
Hilbert space $H^s_e(2\pi)$ for all $s$. 

For notational convenience we also introduce corresponding discrete
sequence spaces $h^s$, $s\geq 0$ and $\ell^2$.
\begin{definition}
  Let $s\geq 0$. The Hilbert space $h^s$ is defined as the space of
  all sequences $a=(a_n)_{n\in\mathbb{N}} $ of complex numbers with
  finite norm $||a||_{h^s}<\infty$, with the discrete $s$-norm
  $||\cdot ||_{h^s}$ defined by
\begin{equation}
  ||a||^2_{h^s} = |a_0|^2+2\sum\limits_{n=1}^\infty |a_n|^2 n^{2s}.
\end{equation} 
and with the natural associated scalar product. We also define $\ell^2
= h^0$.
\end{definition}

The main purpose of this paper is to establish the following theorem.
\begin{theorem}\label{NStheorem}\label{th_1}
  The composition $\tilde{N}\tilde{S}$ defines a bicontinuous operator
  from $H^s_e(2\pi)$ to $H^s_e(2\pi)$ for all $s>0$. Further, this
  operator satisfies a generalized Calder\'on formula
\begin{equation}\label{NSfact}
\tilde{N}\tilde{S}= \tilde J^\tau_0 +\tilde K,
\end{equation}
where ${\tilde K}: H^s_e(2\pi) \rightarrow H^s_e(2\pi)$ is a compact
operator, and where $\tilde J^\tau_0: H^s_e(2\pi) \rightarrow
H^s_e(2\pi)$ is a bicontinuous operator, independent of $k$, with
point spectrum equal to the union of the discrete set
$\Lambda_\infty=\{\lambda_0=-\frac{\ln2}{4},\quad \lambda_n=-\frac{1}{4}-\frac{1}{4n}:n>0\}$ and a certain open set set $\Lambda_s$ which is bounded
away from zero and infinity. The sets $\Lambda_s$ are nested, they
form a decreasing sequence, and they satisfy
$\bigcap_{s>0}\bar{\Lambda}_s=\{-\frac{1}{4}\}$, where
$\bar{\Lambda}_s$ denotes the closure of $\Lambda_s$.  In addition,
the operators
\begin{equation}
  \tilde{S}:\; H^s_e(2\pi) \rightarrow H^{s+1}_e(2\pi)\quad\mbox{and}
\end{equation}
\begin{equation}\label{Ndomain_3}
\tilde{N}:\; H^{s+1}_e(2\pi) \rightarrow H^{s}_e(2\pi)
\end{equation}
are bicontinuous.
\end{theorem}
We thus see that, through introduction of the weight $\omega$ and use
of spaces of even and $2\pi$ periodic functions, a picture emerges for
the open-surface case that resembles closely the one found for
closed-surface configurations: the generalized Calder\'on
relation~\eqref{NSfact} is analogous to the Calder\'on
formula~\eqref{Calderon}, and mapping properties in terms of the
complete range of Sobolev spaces are recovered for $\tilde{S}$ and
$\tilde{N}$, in a close analogy to the framework embodied by
equations~\eqref{ScmappingHs} and~\eqref{NcmappingHs}.

In the remainder of this paper we present a proof of
Theorem~\ref{NStheorem}.  This proof is based on a number of elements,
the first one of which, presented in Section~\ref{three}, concerns the
operator $\tilde J^\tau_0$ in~\eqref{NSfact}---which corresponds, in
fact, to the zero-frequency/straight-arc version of
Theorem~\ref{NStheorem}.

\section{Straight arc at zero frequency: operators $\tilde J_0$ and $\tilde J^\tau_0$\label{three}}

\subsection{Preliminary properties of the operators $\tilde{S}_0$, $\tilde{N}_0$ and other related operators\label{prelim_prop}}
In the case in which $\Gamma$ is the straight-arc $[-1,1]$ and $k=0$,
$\tilde S$ reduces to Symm's operator~\cite{YanSloan,BrunoHaslam}
\begin{equation}\label{S_0}
\tilde{S}_0 [\tilde \varphi] (\theta ) = -\frac{1}{2\pi}\int_0^\pi
\ln|\cos\theta-\cos\theta'| \tilde \varphi( \theta ) d\theta,
\end{equation}
for which the following lemma holds
\begin{lemma}\label{lemma1} The operator $\tilde{S}_0$ maps
  $H^s_e(2\pi)$ into $H^{s+1}_e(2\pi)$, and
\begin{equation}\label{S0reg}
  \tilde{S}_0:\quad H^s_e(2\pi) \rightarrow H^{s+1}_e(2\pi)\quad \mbox{is bicontinuous for all }s\geq 0.
\end{equation}\end{lemma}
\begin{proof}
  It follows from the weak singularity of the kernel in
  equation~\eqref{S_0} that
\begin{equation}\label{S0cont}
  \tilde{S}_0:\quad H^0_e(2\pi) \rightarrow H^{0}_e(2\pi)\quad \mbox{is a continuous operator}. 
\end{equation}
Furthermore, taking into account the well documented diagonal
property~\cite{MasonHandscomb}
\begin{equation}\label{S0Diag}
\tilde{S_0}[e_n] = \lambda_n e_n, \quad \lambda_n= \left\{
\begin{array}{cc}\frac{\ln 2}{2}& n=0\\ \frac{1}{2n}, & n\geq 1
\end{array} \right.
\end{equation}
of Symm's operator in the basis $\{e_n: n\geq 0\}$ of $H^s_e(2\pi)$
defined by 
\begin{equation}\label{basis}
  e_n(\theta) = \cos n\theta,
\end{equation} 
we see that, for every basis element $e_n, n\geq 0$, the operator
$\tilde{S}_0$ coincides with the diagonal operator defined by
 \begin{equation}
   W[f]=\sum\limits_{n\geq 0}\lambda_n f_n e_n \quad \mbox{for} \quad f =\sum\limits_{n\geq0} f_n e_n \in H^s_e(2\pi).
\end{equation}
Clearly the operator $W:\quad H^s_e(2\pi) \rightarrow H^{s+1}_e(2\pi)$
is bicontinuous for all $s\geq 0$, and it is in particular a
continuous operator from $H^0_e(2\pi)$ into $H^0_e(2\pi)$. The
continuous operators $\tilde{S}_0$ and $W$ thus coincide on the dense
set $\{e_n \}$ of $H^0_e(2\pi)$, and they are therefore equal
throughout $H^0_e(2\pi)$. It follows that $\tilde{S}_0 = W$ maps
$H^s_e(2\pi)$ into $H^{s+1}_e(2\pi)$ bicontinuously, and the proof is
complete.
\end{proof}

The corresponding zero-frequency straight-arc version $\tilde{N}_0$ of the
operator $\tilde N$, in turn, is given by
\begin{equation}\label{N0def0}
\tilde{N_0}[\tilde{\psi}](\theta)=\frac{1}{4\pi}\lim\limits_{z\rightarrow
0}\frac{\partial^2}{\partial z^2}\int_{0}^\pi \ln| (\cos\theta -
\cos\theta')^2 + z^2 | \tilde{\psi}(\theta')\sin^2\theta' d\theta',
\end{equation}
which, following~\cite{Kress,ColtonKress1,Monch} we express in the form 
\begin{equation}\label{N0Def}
\tilde{N}_0 = \tilde D_0 \tilde{S}_0 \tilde T_0
\end{equation}
where
\begin{equation}\label{D0Def}
\tilde D_0[\tilde \varphi](\theta) = \frac{1}{\sin\theta}\frac{d\tilde
\varphi (\theta)}{d\theta}
\end{equation} and
\begin{equation}\label{T0Def}
\tilde T_0[\tilde \varphi](\theta) = \frac{d}{d
\theta}\left(\tilde\varphi(\theta) \sin\theta\right).
\end{equation}
(The general curved-arc arbitrary-frequency version of this relation
is presented in Lemma~\ref{Nfact_lemma} below and, for the sake of
completeness, a derivation of the general relation is provided in
Appendix~\ref{Nfact_appendix}.)

Note that, in contrast with the closed-arc case~\cite[p. 117]{Kress},
the expressions~\eqref{N0Def} through~\eqref{T0Def} contain the
vanishing factor $\sin\theta$ and the singular factor $1/\sin\theta$;
in particular, for example, it is not immediately clear that the
operator $\tilde{N}_0$ maps $H^{s+1}_e(2\pi)$ into
$H^{s}_e(2\pi)$. This result is presented in
Corollary~\ref{coro-3}. In preparation for our proofs of that and
other straight-arc zero-frequency results in the following sections,
in the remainder of this section we establish a preliminary continuity
result for the operator $\tilde{D}_0$.
\begin{lemma}\label{D0lemma}
  The operator $\tilde{D}_0$ defines a bounded mapping from
  $H^2_e(2\pi)$ into $H^0_e(2\pi)$.
\end{lemma}

\begin{proof}
  Let $\tilde\varphi=\sum_{n=0}^\infty \tilde\varphi_n e_n$ be an
  element of $H^2_e(2\pi)$. We assume at first that
  $\tilde\varphi_{2p+1} = 0$ for all integers $p\geq 0$. Let $P>0$ and
  $\tilde\varphi^P=\sum_{p=0}^P\tilde\varphi_{2p}e_{2p}$; clearly
  $\tilde\varphi^P$ converges to $\tilde\varphi$ in $H^2_e(2\pi)$ as
  $P\to\infty$. We have
\begin{equation}
\tilde D_0[\tilde\varphi^P](\theta)=-\sum\limits_{p=0}^P 2p \tilde \varphi_{2p}
\frac{\sin(2p\theta)}{\sin \theta}
\end{equation}
In view of the identity
\begin{equation}\label{Un}
\frac{\sin(n+1)\theta}{\sin \theta}=\left\{\begin{array}{ll} \sum\limits_{k=0}^p (2-\delta_{0k})\cos2k\theta, & n=2p \\ 2\sum\limits_{k=0}^p \cos(2k+1)\theta, & n=2p+1, \end{array}\right. 
\end{equation}
(which expressed in terms of Chebyshev polynomials of the first and
second kind is given e.g. in equation (40)~\cite[p. 187]{Bateman} and
problem~3 in~\cite[p. 36]{MasonHandscomb}), we obtain 
\begin{equation}\label{D0dec}
\tilde D_0[\tilde\varphi^P]=-2\sum\limits_{k=1}^{P}
\left(\sum\limits_{p = k}^\infty 2p\tilde\varphi_{2p}^P\right)e_{2k-1}
\end{equation}
where
\[
\tilde\varphi_{2p}^P = \left\{\begin{array}{cc} \tilde\varphi_{2p}^P, &
p\leq P \\ 0, & p> P\end{array}\right. .
\]

The quantity in parenthesis on the right-hand-side of
equation~\eqref{D0dec} can be expressed in terms of the adjoint $C^*$
of the discrete Ces\`aro operator $C$, where $C$ and $C^*$ are given
by
\begin{equation}
C[g](n)=\frac{1}{n+1}\sum\limits_{k=0}^n g_k, \quad  C^*[g](k)=\sum\limits_{p=k}^\infty \frac{g_p}{p+1};
\end{equation}
as there follows from~\cite{BrownHalmosShields}, $C$ and $C^*$ define
bounded operators from $\ell^2$ into $\ell^2$.  We thus re-express
equation~\eqref{D0dec} as
\begin{equation}\label{D0dec2}
\tilde{D}_0[\tilde\varphi^P]=-2\sum\limits_{k=1}^{P}
C^*[g^P](k)e_{2k-1} ,
\end{equation}
where the sequence $g^P$ is given by
$g^P_p=2p(p+1)\tilde{\varphi}^P_{2p}$.  Clearly $g^P$ is an element of
$\ell^2$ and we have
 \begin{equation}\label{gtophi}
   ||g^P||_{\ell^2} =\left(\sum\limits_{p=1}^P |g_p^P|^2\right)^{\frac 12} \leq \left(\sum\limits_{p=1}^P (2p)^4 |\tilde\varphi_{2p}^P|^2\right)^{\frac 12}\leq 2^{1/2} ||\tilde\varphi^P||_{H^2_e(2\pi)}.
 \end{equation}
 In view of the boundedness of $C^*$ as an operator from $\ell^2$
 into $\ell^2$ we obtain
\begin{equation}
||\tilde{D}_0[\tilde\varphi^P]||_{H^0_e(2\pi)}\leq 2 ||C^*||_{\ell^2} ||g^P||_{\ell^2},
\end{equation}
and thus, in view of~\eqref{gtophi}, 
\begin{equation}
||\tilde{D}_0[\tilde\varphi^P]||_{H^0_e(2\pi)}\leq 2^{3/2}||C^*||_{\ell^2} ||\tilde \varphi^P||_{H^2_e(2\pi)}.
\end{equation} 
A similar manipulation on odd-termed sequences can be performed and,
in all, it follows that for any $P>0$ and any element
$\tilde\varphi=\sum_{n=0}^\infty \tilde\varphi_n e_n\in H^2_e(2\pi)$
the corresponding finite-term truncation $\tilde
\varphi^P=\sum_{n=0}^P \tilde \varphi_n e_n$ satisfies
\begin{equation}\label{D0bound}
||\tilde{D}_0[\tilde\varphi^P]||_{H^0_e(2\pi)}\leq K ||\tilde \varphi^P||_{H^2_e(2\pi)}
\end{equation}
for some constant $K$ which does not depend on $P$.

Since $\tilde \varphi^P$ converges in $H^2_e(2\pi)$ to $\tilde
\varphi$ as $P\to\infty$, it follows that $\tilde \varphi^P$ is a
Cauchy sequence in $H^2_e(2\pi)$:
\begin{equation}\label{cauchy}
||\tilde \varphi^P - \tilde \varphi^Q||_{H^2_e(2\pi)}\to 0 \quad
\mbox{as} \quad P,\, Q\to\infty.
\end{equation}
Now, $(\tilde \varphi^P - \tilde \varphi^Q)$ can be viewed as a
finite-term truncation of an element of $H^2_e(2\pi)$ and, thus, the
estimate~\eqref{D0bound} applies to it: we obtain
\begin{equation}\label{D0bound2}
||\tilde{D}_0[\tilde\varphi^P] - \tilde{D}_0[\tilde\varphi^Q]||_{H^0_e(2\pi)}\leq K  ||\tilde \varphi^P - \tilde \varphi^Q||_{H^2_e(2\pi)}.
\end{equation}
Equations~\eqref{cauchy}-\eqref{D0bound2} show that
$\tilde{D}_0[\tilde\varphi^P]$ is a Cauchy sequence in $H^0_e(2\pi)$.
It follows that the sequence $\tilde{D}_0[\tilde\varphi^P]$ converges
in that space as $P\to\infty$ and, in particular, that
$\tilde{D}_0[\tilde\varphi]$ is an element of $H^0_e(2\pi)$. Taking
limit as $P\to\infty$ in equation~\eqref{D0bound}, finally, yields the
inequality
\[
  ||\tilde{D}_0[\tilde\varphi]||_{H^0_e(2\pi)}\leq K ||\tilde\varphi||_{H^2_e(2\pi)}
\]
which establishes the needed boundedness of the operator
$\tilde{D}_0$. The proof is now complete.
\end{proof}

\begin{corollary}\label{N0cor}
  The operator $\tilde{N}_0$ defines a bounded mapping from
  $H^2_e(2\pi)$ into $H^0_e(2\pi)$.
\end{corollary}
\begin{proof}
  The proof follows from Lemma~\ref{D0lemma}, the
  decomposition~\eqref{N0Def}, the continuity of $~\tilde{S}_0$
  established in ~\eqref{S0reg}, and the easily verified observation
  that $\tilde{T}_0$ defines a continuous operator from $H^{s+1}_e$
  and $H^s_e$:
 \begin{equation}\label{T0map}
 \tilde{T}_0:\quad H^{s+1}_e(2\pi) \rightarrow H^s_e(2\pi).
\end{equation}
\end{proof}

A preliminary boundedness result for the composite operator $\tilde
J_0 = \tilde{N}_0\tilde{S}_0$ follows from this Corollary.
 \begin{corollary}\label{J0cor}
   The straight-arc zero-frequency version of the composite operator
   $\tilde N \tilde S$, which is given by
\begin{equation}\label{j0def}
  \tilde J_0 = \tilde N_0\tilde S_0,
\end{equation} 
defines a bounded operator from $H^{1}_e(2\pi)$ into $H^0_e(2\pi)$.
\end{corollary}
In the following section we show that, as stated in Theorem~\ref{th_1}
for the related operator $\tilde J^\tau_0$
(cf. Section~\ref{J_0_tau}), not only does $\tilde J_0$ define a
continuous operator from $H^{1}_e(2\pi)$ into $H^0_e(2\pi)$
(Corollary~\ref{J0cor}): $\tilde J_0$ can also be viewed as a
continuous operator from $H^s_e(2\pi)$ into $H^s_e(2\pi)$ for all
$s\geq 0$.

\subsection{Boundedness of $\tilde{J}_0$ in $H^s_e$ and link with the
  continuous Ces\`aro operator\label{cos_basis}}
The continuity proof presented in this section is based in part on the
following lemma, whose proof relies on use of a certain operator
$\tilde{C}$ related to the {\em continuous} Ces\`aro operator. (Note
that the constructions in Section~\ref{prelim_prop} invoke properties
of the {\em discrete} Ces\`aro operator, instead).
\begin{lemma}\label{Cprop}
  For all $s\geq 0$ the integral operator
\begin{equation}\label{Cfact}
{\tilde C}[\tilde \varphi](\theta)=
\frac{\theta(\pi-\theta)}{\pi\sin\theta}\left[
  \frac{1}{\theta}\int_0^\theta \tilde \varphi(u) du -
  \frac{1}{\pi-\theta}\int_\theta^\pi \tilde \varphi(y) du \right]
\end{equation}
maps $H^s_e(2\pi)$ continuously to itself. Furthermore, 
\begin{equation}\label{Cdef}
{\tilde C}[e_n] (\theta ) = \left\{\begin{array}{ll} 
 0& \text{for $n=0$}\\
\frac{\sin
n\theta}{n\sin\theta} & \text{for $n>0$}, \end{array}\right.
\end{equation}
for all $n\geq 0$
\end{lemma}
\begin{proof}
  The relation~\eqref{Cdef} results from simple manipulations. The
  integral operator on the right-hand side of equation~\eqref{Cfact}
  in turn, can be expressed in terms of the continuous Ces\`aro
  operator
\begin{equation}\label{expr_ces}
C[f](x)=\frac{1}{x}\int_0^x f(u)du=\int_0^1f(xu)du.
\end{equation}
As is known~\cite{BrownHalmosShields}, $C$ is a bounded operator from
$L^2[0,b]$ into $L^2[0,b]$ (the space of square-integrable functions
over $[0,b]$) for all $b>0$. In view of the relations~\eqref{Cdef} it
follows that the operator $\tilde C$ can be extended in a unique
manner as a bounded operator from $H^0(2\pi)$ to $H^0(2\pi)$. Taking
into account equation~\eqref{expr_ces}, further, for each $f\in
C^\infty_0[0,b]$, $m\in \mathbb{N}$ and $x\in[0,b]$ we obtain
\begin{equation}
\left|\frac{\partial^m {C}[f](x)}{\partial x^m}\right|^2\leq \left(\int_0^1
\left|u^mf^{(m)}(xu)\right|du\right)^2 \leq \left(\int_0^1
\left|f^{(m)}(xu)\right|du\right)^2 = \left(C\left[g\right](x)\right)^2
\end{equation}
where $g=\left|f^{(m)}\right|$. Integrating this inequality with
respect to $x$ and taking into account the boundedness of $C$ as an
operator from $L^2$ to $L^2$ there results
\begin{equation}
\int_0^{2\pi}\left|\frac{\partial^m {C}[f](x)}{\partial x^m}\right|^2 dx\leq
M||f^{(m)}||_{L_2[0,2\pi]}^2
\end{equation}
for some constant $M$. It follows easily from this inequality that
$\tilde{C}$ is a continuous operator from $H^m_e(2\pi)$ into
$H^m_e(2\pi)$ for all non-negative integers $m$.  Letting $H^m(2\pi)$
be the space of $2\pi$ periodic functions whose derivatives of order
$k$ are square integrable in any bounded set of the line for all
integers $k\leq m$ (c.f.~\cite{Kress}) we see that $\tilde{C}$ equals
the restriction to $H^m_e(2\pi)$ of some continuous operator
$\tilde{P}: H^m(2\pi)\to H^m(2\pi)$: we may simply take, for example,
$\tilde{P}$ to equal $\tilde{C}$ on the subspace of even functions and
to equal 0 on the space of odd functions.  In view of the Sobolev
interpolation result (see e.g.~\cite[Theorem 8.13]{Kress}),
$\tilde{P}$ defines a continuous operator from $H^s(2\pi)$ to
$H^s(2\pi)$ for all $s\geq 0$, and thus, by restriction of $\tilde{P}$
to the susbspace of even and period functions we see that $\tilde{C}$
is a continuous operator from $H^s_e(2\pi)$ to $H^s_e(2\pi)$ for all
$s\geq 0$.
\end{proof}\\

Our main result concerning the operator $\tilde{J}_0$ is given in the
following lemma.
\begin{lemma}\label{J0bounded} The composition 
  $\tilde J_0= \tilde N_0\tilde S_0$ defines a bounded operator from
  $H^s_e(2\pi)$ into $H^s_e(2\pi)$ for all $s \geq 0$.
\end{lemma}
\begin{proof}
We first evaluate the action of $\tilde J_0$ on the basis $\{e_n:n\geq 0\}$.
The case $n=0$ is straightforward: in view of~\eqref{S0Diag}
and~\eqref{N0Def} we have
\begin{equation}\label{0case}
\tilde J_0[e_0](\theta) =-\frac{\ln2}{4}.
\end{equation}
For $n\geq 0$, in turn, expanding~(\ref{T0Def}) we obtain
\begin{equation}\label{T0cos}
\begin{split}
\tilde T_0[e_n](\theta) = \cos\theta \cos n\theta-n\sin n\theta \sin \theta
\\ =\frac{\cos(n+1)\theta+\cos(n-1)\theta}{2}+n\frac{\cos(n+1)\theta
-\cos(n-1)\theta}{2}
\end{split}
\end{equation}
which, for $n\geq 2$, in view of~\eqref{S0Diag} yields, upon
application of $\tilde{S}_0$,
\begin{equation}\label{S0T0}
\tilde{S}_0 \tilde T_0[e_n](\theta) =
\frac{\cos(n+1)\theta}{4(n+1)}+\frac{\cos(n-1)\theta}{4(n-1)}+n\left(\frac{\cos(n+1)\theta}{4(n+1)}
-\frac{\cos(n-1)\theta}{4(n-1)}\right)\quad,\quad (n\geq 2).
\end{equation}
In view of~\eqref{N0Def}-\eqref{D0Def}, for $n \geq 2$ we thus obtain
the relation
 \begin{equation}\label{N0en}
   \tilde{N}_0[e_n](\theta)=-\cos\theta\frac{\sin
     n\theta}{2\sin\theta}-\frac{n}{2}\cos n\theta, 
 \end{equation}
 which, as it is easily verified, also holds for $n=1$.  Using this
 relation in conjunction with~\eqref{S0Diag} and~\eqref{0case} it
 follows that
 \begin{equation}\label{N0S0cos}
   \tilde J_0[e_n](\theta)=\left\{\begin{array}{ll}-\frac{\ln 2}{4}, & n =0 \\
       -\cos\theta\frac{\sin n \theta}{4n\sin \theta} -\frac{\cos n\theta}{4},& n>0\end{array}\right.
 \end{equation}

It can be easily verified that the operator $\tilde W_0$  defined by 
\begin{equation}\label{Jfact}
\tilde W_0[\varphi](\theta)=-\frac{\tilde \varphi(\theta)}{4} -
\frac{\cos\theta}{4}{\tilde C[\tilde \varphi](\theta)}+\frac{1-\ln 2 }{4\pi}\int_0^\pi\tilde\varphi(\theta)d\theta,
\end{equation}
reduces to the right-hand side of~\eqref{N0S0cos} when evaluated on
the basis functions: $\tilde{J}_0[e_n]=\tilde W_0[e_n],$ for all
$n\geq 0$ (the last term in equation~\eqref{Jfact} is obtained by
collecting the zero-th order terms, and explicitly expressing the
zero-th order coefficient of $\tilde{\varphi}$ as an integral). In
view of Lemma~\ref{Cprop} we see that the operator $\tilde W_0$
defines a bounded mapping from $H^s_e(2\pi)$ into $H^s_e(2\pi)$ for
all $s>0$. We conclude the equality of $\tilde W_0$ and $\tilde{J}_0$
from the continuity of $\tilde{J}_0$ established in
Corollary~\ref{J0cor}, and the lemma follows.
\end{proof}

From the relationship $\tilde{N}_0 = \tilde{J}_0 \tilde{S}_0^{-1}$ we
immediately obtain the following corollary.
\begin{corollary} \label{coro-3}
  The operator $\tilde{N}_0:\quad H^{s+1}_e(2\pi) \rightarrow
  H^s_e(2\pi)$ is continuous.
\end{corollary}

\begin{remark}\label{rem_2}
  The decomposition~\eqref{Jfact} superficially resembles the
  classical closed-surface Calder\'on formula~\eqref{Calderon}, as it
  expresses the operator $\tilde W_0 = \tilde J_0 = \tilde N_0\tilde
  S_0$ as the sum of $-I/4$ and an additional operator. As shown in
  Section~\ref{point_sp} below, however, the operator ${\tilde C}:
  H^s_e(2\pi) \to H^s_e(2\pi)$ which appears in~\eqref{Jfact} is not
  compact---and thus the Fredholm theory cannot be applied to
  establish the continuous invertibility of $\tilde J_0 = \tilde W_0$
  merely on the basis of the decomposition~\eqref{Jfact}.  The results
  of Section~\ref{inv_j0} nevertheless do establish that the operator
  $\tilde J_0 = \tilde N_0\tilde S_0$ is
  bicontinuous. Section~\eqref{point_sp} then provides a description
  of the spectrum of $\tilde J_0$, and, in preparation for the proof
  of Theorem~\ref{th_1}, Section~\ref{J_0_tau} extends all of these
  results to the operator $\tilde{J}_0^\tau$.
  
\end{remark}

\subsection{Invertibility of $\tilde J_0$ \label{inv_j0}}
We now proceed to show that the continuous operator $\tilde J_0 :
H^s_e(2\pi)\to H^s_e(2\pi)$ admits a (bounded) inverse. Noting that
the decomposition~\eqref{N0Def} is not directly invertible on a term
by term basis ($\tilde T_0$ and $\tilde D_0$ are not invertible) we
first state and proof two lemmas concerning the mapping properties of
the operators $\tilde T_0$ and $\tilde D_0$.
\begin{lemma} The operators $\tilde{C}$ and $\tilde{T}_0$ satisfy
\begin{equation}\label{T0C}
  \tilde{T}_0 \tilde C[ e_n ]= \left\{\begin{array}{ll} 0,& n=0 \\
e_n,& n> 0. \end{array}\right.
\end{equation}
and 
\begin{equation}\label{CT0}
  \tilde{C} \tilde T_0[\tilde \varphi]= \tilde \varphi\quad \mbox{for all}\quad \tilde \varphi\in H^s_e(2\pi).
\end{equation}
\end{lemma}
\begin{proof}
In view of~\eqref{T0Def} and~\eqref{Cdef} we
clearly obtain equation~\eqref{T0C}, while equation~\eqref{CT0} follows
immediately from~\eqref{Cfact}. 
\end{proof}\\

\begin{lemma}
  For all $s\geq 2$ the operator $\tilde{D}_0$ can be expressed in the
  form
\begin{equation}\label{DDinv}
  \tilde{D}_0 = -\frac{1}{4}\tilde{C}\left( \tilde{S}_0^{-1}\right)^2.
\end{equation}
\end{lemma}
\begin{proof}
In view of~\eqref{D0Def} we have
\begin{equation}
\tilde{D}_0[e_n](\theta)=\left\{  \begin{array}{cc} 0, &n=0\\ -n\frac{\sin n \theta}{\sin\theta}, & n\geq 1. \end{array} \right.
\end{equation} 
In view of the continuity of the various operators involved in
relevant Sobolev spaces (Lemmas~\ref{lemma1},~\ref{D0lemma}
and~\ref{Cprop}) and the density of the basis $\{e_n\}$ in
$H^s_e(2\pi)$, equation~\eqref{DDinv} follows from~\eqref{S0Diag}
and~\eqref{Cdef}.
\end{proof}

\begin{corollary} 
For all $s>0$, the operator $\tilde{D}_0$ defines a bounded mapping from
$H^{s+2}_e(2\pi)$ into $H^s_e(2\pi)$.
\end{corollary}

\begin{lemma} \label{alt-lemma} 
For each integer $n\geq 2$ we have
\begin{equation}\label{CST}
\tilde{C}\tilde{S}_0\tilde{T}_0[e_n](\theta) = \frac {\cos\theta}{2}
\tilde{C}\left[\frac{n}{1-n^2}\ e_n\right](\theta) -\frac {n\cos n\theta}{2(1-n^2)}.
\end{equation}
\end{lemma}
\begin{proof}
From the easily established identity
\[
\tilde{S}_0\tilde{T}_0[e_n] = \frac 14 ( e_{n+1} - e_{n-1})\quad (n\geq 1),
\]
using equation~\eqref{Cdef} we obtain the relation
\begin{equation}\label{cancellation}
\tilde{C}\tilde{S}_0\tilde{T}_0[e_n] = \frac 12 \left [ \frac{\sin
    n\theta}{\sin \theta}\frac{\cos \theta}{1-n^2} -\frac {n\cos
    n\theta}{1-n^2}\right]\quad (n\geq 2)
\end{equation}
which, via an additional application of equation~\eqref{Cdef} yields
the desired equation~\eqref{CST}.
\end{proof}
\begin{corollary} \label{alt-cor} The composition
  $\tilde{C}\tilde{S}_0\tilde{T}_0$ which, in view of
  Lemma~\ref{lemma1}, Lemma~\ref{Cprop} and equation~\eqref{T0Def},
  defines a continuous operator from $H^s_e(2\pi)$ to $H^s_e(2\pi)$
  for $s\geq 1$, can in fact be extended in a unique fashion to an
  operator defined on $H^s_e(2\pi)$ for each $s\geq 0$.  For all
  $s\geq 0$, further, these extended maps enjoy additional regularity:
  they can be viewed as continuous operators from $H^s_e(2\pi)$ into
  $H^{s+1}_e(2\pi)$.
\end{corollary} 
\begin{proof}
  The proof follows by consideration of
  equation~\eqref{cancellation}. A cancellation of the form $n\sin
  (n+1)\theta - n\sin (n+1)\theta = 0$, which occurs in the process of
  evaluation of the right hand side of equation~\eqref{cancellation},
  underlies the additional regularity of the mapping
  $\tilde{C}\tilde{S}_0\tilde{T}_0$.
\end{proof}

We can now obtain the inverse of the operator $\tilde{J}_0$.
\begin{lemma}\label{J0Inverse} The continuous operator 
  $\tilde J_0: H^s_e(2\pi) \to H^s_e(2\pi)$ ($s\geq 0$) which,
  according to equations~\eqref{N0Def} and~\eqref{j0def} is given by
\begin{equation}\label{J_exp}
\tilde J_0=\tilde D_0 \tilde{S}_0 \tilde T_0\tilde{S}_0,
\end{equation}
is bijective, with (continuous) inverse $\tilde J_0^{-1}:H^s_e(2\pi)
\to H^s_e(2\pi)$ given by
\begin{equation}\label{I0Def}
\tilde J_0^{-1}=-4\tilde{S}_0^{-1}\tilde{C}\tilde{S}_0\tilde{T}_0
\end{equation}
for $s\geq 2$, and given by the unique continuous extension of the
right hand side of this equation for $2 > s\geq 0$.

\end{lemma}
\begin{proof}\\
  Since the rightmost factor $\tilde{S}_0$ in equation~\eqref{J_exp}
  is a diagonal operator, we consider the next operator from the right
  in this product, namely, $\tilde T_0$, which in view of
  equations~\eqref{T0C} and~\eqref{CT0}, admits $\tilde{C}$ as a
  ``partial'' inverse. Since the next factor $\tilde{S}_0$ from the
  right is, once again, a diagonal operator, we consider next the
  leftmost factor in equation~\eqref{J_exp}: the operator $\tilde
  D_0$, a decomposition of which was provided in
  equation~\eqref{DDinv}.  In sum, to obtain the inverse of $\tilde
  J_0$ we proceed as follows: multiplying $\tilde J_0$ on the right by
  $\tilde{S}_0^{-1}\tilde{C}$ we obtain an operator that maps $e_0$ to
  $0$ and $e_n$ to $ \tilde{D}_0\tilde{S}_0[e_n]$. Thus,
  considering~\eqref{CT0} and~\eqref{DDinv}, we further multiply on
  the right by $-4\tilde{S}_0\tilde T_0$ and we obtain the operator
\begin{equation}\label{inv_left}
-4\tilde J_0\tilde{S}_0^{-1}\tilde{C}\tilde{S}_0\tilde T_0
\end{equation}
which, in view of the fact that the image of $\tilde{S}_0\tilde T_0$
is orthogonal to $e_0$ (as it follows easily from
equations~\eqref{S0Diag} and~\eqref{T0Def}) maps $e_n$ to
$-4\tilde{D}_0\left( \tilde{S}_0\right)^2\tilde T_0[e_n]$ for all
$n\geq 0$. But, in view of~\eqref{DDinv}, this quantity equals
$\tilde{C} \tilde T_0[e_n]$ which, according to~\eqref{CT0}, equals
$e_n$. In other words, the operator~\eqref{inv_left}, which is a
continuous operator from $H^s_e(2\pi)$ to $H^{s-1}_e(2\pi)$ ($s\geq
1$), maps $e_n$ to $e_n$ for $n=0,1,2\dots$---and, thus,
\begin{equation}\label{I0-2}
\tilde{I}_0 = -4\tilde{S}_0^{-1}\tilde{C}\tilde{S}_0\tilde T_0
\end{equation}
is a right inverse of $\tilde J_0$, that is
\begin{equation}\label{rght-inv-sgr1}
  \tilde J_0\tilde{I}_0 = I,
\end{equation} 
at least for $s\geq 1$.

Conversely, since in view of equations~\eqref{J_exp} and~\eqref{DDinv}
$\tilde{J}_0$ can be expressed, for $s\geq 1$, in the form
\begin{equation}\label{J0dense}
\tilde{J}_0=-\frac{1}{4}\tilde{C}\tilde{S}_0^{-1}\tilde{T}_0\tilde{S}_0,
\end{equation}
for $s\geq 2$ we have
\begin{equation}\label{I0J0}
  \tilde I_0 \tilde J_0 = \tilde S_0^{-1} \tilde C \tilde S_0
  \tilde T_0 \tilde C \tilde S_0^{-1} \tilde T_0\tilde{S}_0. 
\end{equation}
Now, as noted above, the image of $\tilde{T}_0$ is orthogonal to
$e_0$, and thus, since $\tilde S_0$ is a diagonal operator, the same
is true of the operator $\tilde S_0^{-1} \tilde T_0 \tilde
S_0$. Equation~\eqref{T0C} can therefore be used directly to obtain
\begin{equation}
\tilde T_0 \tilde C \tilde S_0^{-1} \tilde T_0 \tilde S_0[ e_n ]=  \tilde
S_0^{-1} \tilde T_0 \tilde S_0[ e_n ], \quad  \mbox{for all}\quad n \geq 0.
\end{equation}
Clearly then, equation~\eqref{I0J0} can be reduced to
\begin{equation} 
\tilde{I}_0\tilde{J}_0=\tilde{S}_0^{-1}\tilde C \tilde{T}_0\tilde{S}_0,
\end{equation}
and making use of~\eqref{CT0}, we finally obtain
\begin{equation}\label{lft-inv}
\tilde{I}_0\tilde{J_0}=I, 
\end{equation}
as desired, thus establishing the invertibility of $\tilde{J}_0$ at
least for $s\geq 2$.  The boundedness of $\tilde{I}_0 =
\tilde{J}_0^{-1}$ for $s\geq 2$ follows in view of
Remark~\ref{inv_is_cont} (continuity of inverses of continuous linear
maps) or, otherwise, directly from equation~\eqref{I0Def},
Corollary~\ref{alt-cor} and Lemma~\ref{lemma1}. To treat the case
$2>s\geq 0$, finally, we note that by Corollary~\ref{alt-cor} and
Lemma~\ref{lemma1} $\tilde{I}_0$ can be extended in a unique fashion
as a continuous mapping from $H^s_e(2\pi)$ to $H^s_e(2\pi)$ for all
$s\geq 0$, and that by Lemma~\ref{J0bounded} $\tilde{J_0}$ is
continuous mapping from $H^s_e(2\pi)$ to $H^s_e(2\pi)$ for all $s\geq
0$. The $s\geq 2$ relations~\eqref{rght-inv-sgr1} and~\eqref{lft-inv}
thus extend to all $s\geq 0$ by density of $H^2_e(2\pi)$ in
$H^s_e(2\pi)$ ($2>s\geq 0$), and the proof is thus complete.
\end{proof}

\begin{corollary}\label{N0mapping}
  For all $s\geq 0$, the operator
  $\tilde{N}_0=\tilde{J}_0\tilde{S}_0^{-1}$ defines a bicontinuous
  mapping from $H^{s+1}_e(2\pi)$ to $H^s_e(2\pi)$.
\end{corollary}
\begin{proof}
This follows directly from equation~\eqref{S0reg}, equation~\eqref{j0def}, and
Lemmas~\ref{J0bounded} and ~\ref{J0Inverse}.
\end{proof}

\subsection{Point Spectrum of $\tilde J_0$\label{point_sp}}
Having established boundedness and invertibility, we conclude our
study of the operator $\tilde J_0$ by computing its eigenvalues.
\begin{lemma}\label{Jeigenvalues}
  For any $s>0$, the point spectrum $\sigma_s$ of $\tilde J_0 :
  H^s_e(2\pi) \to H^s_e(2\pi)$ can be expressed as the union
\begin{equation}
\sigma_s = \Lambda_s \cup\Lambda_\infty,
\end{equation}
where $\Lambda_\infty$ is the discrete set 
\begin{equation}\label{lambdaInfinity}
  \Lambda_\infty=\left\{\lambda_n: n=0,1,\dots, \infty\right\},\; \lambda_n=\left\{\begin{array}{ll}-\frac{\ln 2 }{4} , & n = 0 \\
      -\frac{1}{4} - \frac{1}{4n},& n>0 ,
\end{array}
\right.
\end{equation}
and where $\Lambda_s$ is the open bounded set
\begin{equation}\label{lambda_s}
  \Lambda_s=\left\{ \lambda=(\lambda_x+i\lambda_y)\in\mathbb{C}\; : \;4s+2 <
    \frac{-\left(\lambda_x+\frac{1}{4}\right)}{(\lambda_x+\frac{1}{4})^2+\lambda_y^2}\right\}.
\end{equation}
\end{lemma}
\begin{proof}
  We start by re-expressing equation~\eqref{N0S0cos} as
\begin{equation}\label{Jc0}
  \tilde J_0[e_n](\theta)=\left\{\begin{array}{ll}-\frac{\ln 2}{4}&n=0\\-\frac{\sin (n+1)\theta}{4n\sin\theta}+\frac{\cos
        n\theta}{4n}-\frac{\cos n \theta}{4},& n>0.
\end{array}\right.
\end{equation}
Then, making use again of~\eqref{Un}
we obtain
\begin{equation}\label{J0exp}
\tilde J_0[ e_n] = \left\{ \begin{array}{ll} \lambda_n e_n -
  \frac{1}{2n}\sum\limits_{k=0}^{p-1}(1-\frac{\delta_{0k}}{2})e_{2k},& n =
  2p,\; p\geq0\\ \lambda_n e_n - \frac{1}{2n}\sum\limits_{k=0}^{p-1}e_{2k+1},&
  n = 2p+1,\; p\geq0\\ \end{array}\right.,
\end{equation}
where the diagonal elements $\lambda_n$ are defined in
equation~\eqref{lambdaInfinity}.  Clearly, $\tilde{J}_0$ takes the
form of an upper-triangular (infinite) matrix whose diagonal terms
$\lambda_n$ define eigenvalues associated with eigenvectors $v_n$,
each one of which can be expressed in terms of a finite linear
combination of the first $n$ basis functions: $v_n=\sum_{k=0}^n c_k^n
e_k$. In particular, for all $n\in \mathbb{N}$, $v_n \in
H^s_e[0,2\pi]$ for all $s > 0$. This shows that the set
$\Lambda_\infty$ of diagonal elements defined in
equation~\eqref{lambdaInfinity} is indeed contained in $\sigma_s$ for
all $s>0$.

As is well known, an upper triangular operator in an
infinite-dimensional space can have eigenvalues beyond those
represented by diagonal elements. As shown in
\cite[Th. 2]{BrownHalmosShields}, for instance, the point spectrum of
the upper-triangular bounded operator
\begin{equation}
C^*[a](n) = \sum\limits_{k=n}^\infty \frac{a_k}{k+1},
\end{equation}
(the adjoint of the discrete Ces\`aro operator $C$) is the open disc
$|\lambda-1|<1$. A similar situation arises
for our operator $\tilde J_0$.

To obtain the full point spectrum of the operator $\tilde J_0$ let
$\lambda \in \mathbb{C}$ and $f=\sum_{k=0}^\infty f_n e_n$ be such
that $\tilde{J}_0[f]=\lambda f$.  It follows from~\eqref{J0exp} that
the coefficients $f_n$ satisfy the relation
\begin{equation}\label{recfn}
  (-\frac{1}{4}-\frac{1}{4n})f_{n} -\frac{1}{2}
  \sum\limits_{k=1}^\infty \frac{f_{n+2k}}{n+2k}=\lambda f_{n}\quad,\quad
  n \geq 1,
\end{equation}
along with 
\begin{equation}
  (-\frac{\ln 2}{4})f_{0} -\frac{1}{4}
  \sum\limits_{k=1}^\infty \frac{f_{2k}}{2k}=\lambda f_{0}\quad,\quad
  n = 0.
\end{equation}
Equation~\eqref{recfn} is equivalent to 
\begin{equation}
\frac{1}{2}\sum\limits_{k=1}^\infty \frac{f_{n+2k}}{n+2k} = f_n(
-\frac{1}{4n}-\frac{1}{4}-\lambda),\quad n\geq 1,
\end{equation}
which, by subtraction, gives
\begin{equation}
\frac{1}{2}\frac{f_{n+2}}{n+2}  = f_n(
-\frac{1}{4n}-\frac{1}{4}-\lambda)-f_{n+2}(-
\frac{1}{4(n+2)}-\frac{1}{4}-\lambda),\quad n \geq 1.
\end{equation}
Therefore, the coefficients of $f$ must satisfy
\begin{equation}\label{fRec}
\left\{\begin{array}{rlc}
f_{n+2}=&f_n\left(\frac{ \frac{z}{2} + \frac{1}{n} }{\frac{z}{2}
  -\frac{1}{(n+2)}  }\right),&\quad n \geq 1\\
\frac{1}{4}\sum\limits_{k=1}^\infty \frac{f_{2k}}{2k} =& f_0( -\frac{\ln
  2}{4}-\lambda),&\quad n=0.
\end{array}\right.
\end{equation}
where, in order to simplify the notations, we write
 \begin{equation}\label{zdef}
z = 8\lambda  +2.
\end{equation}
It is clear from equation~\eqref{fRec} that the zero-th coefficient is
determined by the coefficients of even positive orders,
and that the sequence $f_n$ for $n\geq 1$ is entirely determined by
$f_1$ and $f_2$. 

Clearly, there are no elements of the point spectrum for which $Re(z)
\geq 0 $, since for such values of $z$ the resulting sequence $f_n$ is not
square summable (that is, $\sum |f_n|^2=\infty$). 
Note that the set of vectors $\{v_n\}$ associated
with the discrete eigenvalues $\lambda_n=-\frac{1}{4}-\frac{1}{4n}$,
in turn, are recovered by setting $z=-\frac{2}{n}$.  To determine all
of the elements of the point spectrum with $Re(z)<0$ we study
separately the odd and even terms in the sequence~\eqref{fRec}. We
start with the sequence $q_n=f_{2n}$, which satisfies the recurrence
relationship
 \begin{equation}
q_{n+1}=q_n \left( \frac{ z + \frac{1}{n}}{z -\frac{1}{n+1}}\right),\quad n
\geq 1.
 \end{equation}
Let $z= -x+iy$ with $x > 0$, and assume without loss of generality, that $q_1
= 1 $. Then
\begin{equation}
  q_n = \left(\frac{z-1}{z-\frac{1}{n}}\right)\prod\limits_{k=1}^{n-1}
  \left( \frac{z + \frac{1}{k}}{z -\frac{1}{k}} \right),\quad n\geq 1,
\end{equation}
and it follows that
\begin{equation}\label{logqn}
\begin{split}
\ln|q_n| = \ln\left|\frac{z-1}{z- \frac{1}{n}}\right|+\frac{1}{2}\sum\limits_{k=1}^{n-1} \ln\left( \frac{(x-\frac{1}{k})^2 + y^2}{ (x+\frac{1}{k})^2 + y^2} \right)\\
=  \ln\left|\frac{z-1}{z - \frac{1}{n}}\right|+\frac{1}{2}\sum\limits_{k=1}^{n-1}\ln\left( \frac{1-r(x,y,k))}{1+r(x,y,k)}   \right)
\end{split}
\end{equation}
where
\begin{equation}\label{rxyk}
r(x,y,k)=\frac{2x}{k(x^2+y^2+\frac{1}{k^2})}.
\end{equation}
For large $k$, we have
\begin{equation}
\ln \left( \frac{1-r(x,y,k)}{1+r(x,y,k)}\right)= -\frac{4x}{k(x^2 + y^2) }+O(\frac{1}{k^3}),\; k \rightarrow \infty,
\end{equation}
and thus
\begin{equation}\label{qnAsympt}
   \ln|q_n|= -\frac{2x}{x^2+y^2}\ln n + M + O(\frac{1}{n}),
\end{equation}
where $M$ is a constant.  The absolute value of $q_n$
is thus asymptotically given by
\begin{equation}
|q_n| = O\left( \frac{1}{n^{\frac{2x}{x^2+y^2}}}\right)
\end{equation}
as $n\to\infty$.  It follows that, for any $s>0$, the set of points $(x,y)$
in the half plane such that the sequence $\sum n^{2s} |q_n|^2<\infty$ is
exactly defined by the equation
\begin{equation}\label{convergence_condition}
  2s - \frac{4x}{x^2 + y^2 } < -1.
\end{equation}
The analysis for the odd-term sequence $p_n=f_{2n+1}$ can be carried
out similarly, since
\begin{equation}
  p_{n+1}=p_n\left(\frac{z+\frac{1}{n+\frac{1}{2}}}{z-\frac{1}{n+1+\frac{1}{2}}}\right),\end{equation}
which essentially amounts to replacing $k$ by $k+\frac{1}{2}$ in
equations~\eqref{logqn} and~\eqref{rxyk}. The convergence
condition~\eqref{convergence_condition} thus applies to $p_n$ as well, and 
it follows,  in view of equation~\eqref{zdef}, that the set $\Lambda_s$
defined by~\eqref{lambda_s} contains all the eigenvalues of $\tilde{J}_0$ not contained in
$\Lambda_\infty$. 
\end{proof}
\begin{corollary}
The operator $\tilde{C}\; :H^s_e(2\pi) \to H^s_e(2\pi)$ is not compact.
\end{corollary}
\begin{proof}
  This follows from the decomposition~\eqref{Jfact} of $\tilde{J}_0$
  and the fact that $\tilde{J}_0$ admits a spectrum that is not
  discrete.
\end{proof}
\begin{remark}
Using polar coordinates $(r,\theta)$ around the point $(-\frac{1}{4},0)$ it
is easy to check that
\begin{equation}\label{polarLambda}
  \Lambda_s = \left\{ (\lambda_x+i\lambda_y)\in\mathbb{C}: \lambda_x+\frac{1}{4}=r\cos\theta,\, \lambda_y= r\sin\theta,\, 0<r <
    -\frac{\cos \theta}{4s +2 },\, \theta \in \left[\frac{\pi}{2},
      \frac{3\pi}{2}\right] \right\}.
\end{equation}
Clearly then, for $s>s'$, $\Lambda_s \varsubsetneq \Lambda_{s'}$, and we
have $\bigcap_{s>0} \Lambda_s =\varnothing$, while the intersection of the
closures is given by $\bigcap_{s>0} \bar{\Lambda}_s=
\{-\frac{1}{4}\}$. Also, for all $s>0$, ${\rm dist}(
\sigma_s,0)=-\frac{1}{4}$, and $\max_{\lambda\in\sigma_s}|\lambda| \leq
\frac{3}{4}$.  It therefore follows that $\sigma_s$ is bounded away from the
zero and infinity. In view of Theorem~\ref{th_1} and Section~\ref{J_0_tau},
this is a fact of great significance in connection with the numerical
solution of equations~\eqref{Stilde} and~\eqref{Ntilde} by means of
Krylov-subspace iterative linear-algebra techniques;
see~\cite{BrunoLintner2} for details.
\end{remark}
\subsection{The operator $\tilde J^\tau_0$\label{J_0_tau}}
In our proof of Theorem~\ref{NStheorem} we need to consider not $\tilde J_0$
but a closely related operator, namely
\begin{equation}\label{j_0_tau_def}
\tilde J^\tau_0 = \tilde N^\tau_0 \tilde S^\tau_0
 \end{equation}
where defining (in a manner consistent with equation~\eqref{Z_op} below)
$\tilde{Z}_0[\gamma](\theta)=\gamma(\theta)\tau(\cos\theta)$, we have set
\begin{equation}\label{s_tau_def}
\tilde{S}_0^\tau[\gamma]=\tilde{S}_0\tilde{Z}_0[\gamma ],
\end{equation}
and
\begin{equation}\label{n0_tau_def}
\tilde{N}^\tau_0[\gamma] = \tilde{Z}_0^{-1}\tilde{N}_0[\gamma].
\end{equation}
It is easy to generalize equation~\eqref{S0reg},
Corollary~\ref{N0mapping} and Lemmas~\ref{J0bounded}
through~\ref{Jeigenvalues} to needed corresponding results for
$\tilde{S}_0^\tau$, $\tilde{N}_0^\tau$ and $\tilde J^\tau_0$; these
are given in the following Theorem.
\begin{theorem}\label{j_0_tau_thm} Let $s\geq 0$. Then,
\begin{itemize}
\item[(i)] The operator $\tilde{S}_0^\tau\;\,:\; H^{s}_e(2\pi)\to
  H^{s+1}_e(2\pi)$ is bicontinuous. 
\item[(ii)] The operator $\tilde{N}_0^\tau\;\,:\; H^{s+1}_e(2\pi)\to
  H^s_e(2\pi)$ is bicontinuous. 
\item[(iii)] The operator $\tilde J^\tau_0\;\,:\;H^s_e(2\pi)\to H^s_e(2\pi)$ is bicontinuous.
\item[(iv)] The point spectrum of $\tilde J^\tau_0:\; H^s_e(2\pi)\to
H^s_e(2\pi)$ is equal to the point spectrum $\sigma_s$ of $\tilde J_0$.
\end{itemize}
\end{theorem}
\begin{proof}
In view of~\eqref{s_tau_def}, ~\eqref{n0_tau_def}, the ensuing relation
\begin{equation}\label{j_0_tau_gen}
\tilde J^\tau_0 = \tilde{Z}_0^{-1}\tilde J_0\tilde{Z}_0,
\end{equation}
and the fact that $\tau$ is smooth and non-vanishing, the proof of
points~{\em(i)},~{\em(ii)} and~{\em(iii)} is immediate.
Equation~\eqref{j_0_tau_gen} also shows that $(\lambda,v)$ is an
eigenvalue-eigenvector pair for $\tilde J_0$ if and only if
$(\lambda,\tilde{Z}_0^{-1}[v])$ is an eigenvalue-eigenvector pair for
$\tilde{J}_0^\tau$, and point~{\em(iv)} follows as well.
\end{proof}

\section{General Properties of the  Operators $\tilde S$ and
$\tilde N$ \label{SGeneral}} The proof of Theorem~\ref{NStheorem} results
from a perturbation argument involving Theorem~\ref{j_0_tau_thm} and the
results established in this section on the regularity and invertibility of
the operators $\tilde S$ and $\tilde N$ defined by equations~\eqref{ssdef}
and~\eqref{NNdef}.
\subsection{Bicontinuity of the operator $\tilde{S}$}
We seek to show that for all $s\geq 0$ the operator $\tilde{S}$
defined in equation~\eqref{ssdef} is a bicontinuous mapping between
$H^s_e(2\pi)$ into $H^{s+1}_e(2\pi)$. This is done in
Lemmas~\ref{S_bounded} and~\ref{Sinverse} below.
\begin{lemma}\label{S_bounded}
  Let $s\geq 0$. Then $\tilde{S}$ defines a bounded mapping from
  $H^s_e(2\pi)$ into $H^{s+1}_e(2\pi)$. Further, the difference
  $\tilde{S} - \tilde{S}^\tau_0$ (see equation~\eqref{s_tau_def})
  defines a continuous mapping from $H^s_e(2\pi)$ into
  $H^{s+3}_e(2\pi)$.
\end{lemma}
\begin{proof}
In view of equation~\eqref{Gk_def} and the expression
\[
H_0^1(z) = \frac{2i}{\pi}J_0(z)\ln(z) + R(z)
\]
for the Hankel function in terms of the Bessel function $J_0(z)$, the
logarithmic function and a certain entire function $R$, the kernel of
the operator $S_\omega$ (equation~\eqref{sOmegaDef}) can be cast in
the form
\begin{equation}
G_k(\mathbf{r}(t),\mathbf{r}(t'))=A_1(t,t')\ln|t-t'|+A_2(t,t'),
\end{equation}
where $A_1(t,t')$ and $A_2(t,t')$ are smooth functions. Further, since $J_0(z)$ is given by a series in powers of $z^2$, it follows that for
all $m\in \mathbb{N}$, the function $A_1$ can be
expressed in the form
\[
A_1(t,t')=-\frac{1}{2\pi} + \sum_{n=2}^{m+3} a_n(t)(t'-t)^n +
(t-t')^{m+4}\Lambda_{m+3}(t,t'),
\] 
where $\Lambda_{m+3}(t,t')$ is a smooth function of $t$ and $t'$. The
operator $\tilde{S}$ in equation~\eqref{ssdef} can thus be expressed in the
form
\begin{equation}\label{Sfact}
\begin{split}
\tilde{S}[\tilde \varphi](\theta)=\tilde{S}_0^\tau[\tilde\varphi](\theta) +
\sum\limits_{n=2}^{m+3} a_n(\cos\theta)\int_0^\pi(\cos\theta'- \cos\theta)^n
\ln|\cos\theta-\cos\theta'|\tilde\varphi(\theta')\tau(\cos\theta')d\theta'\\+\int_0^\pi
A_3(\cos\theta,\cos\theta')\tilde\varphi(\theta')\tau(\cos\theta')d\theta',
\end{split}
\end{equation}
where $A_3(\cos\theta,\cos\theta')$, which contains a logarithmic factor,
belongs to $C^{m+3}([0,2\pi]\times [0,2\pi])$.

Clearly, for $n\geq 2$, the second derivative $d^2/d\theta^2$ of the product
$(\cos\theta'- \cos\theta)^n \ln|\cos\theta-\cos\theta'|$ can be expressed
as a product $P_1(\cos\theta,\cos\theta')\ln|\cos\theta-\cos\theta'|+
P_2(\cos\theta,\cos\theta')$ where $P_1(t,t')$ and $P_2(t,t')$ are
polynomials. Collecting terms with the common factor $\cos^\ell\theta'$ we
then obtain
\begin{equation}\label{Sfact_2}
\frac{d^2}{d\theta^2}\left(\tilde{S} -
\tilde{S}_0^\tau\right)[\tilde\varphi](\theta) =\sum\limits_{\ell=0}^{m+1}
b_\ell(\cos\theta)\tilde S_0 \tilde Z_\ell [\tilde \varphi](\theta) +
\int_0^\pi A_4(\cos\theta,\cos\theta')\tilde\varphi(\theta')\tau(\cos\theta')d\theta',
\end{equation}
where $b_\ell(\cos\theta)$ is an even smooth function, where the
operator $\tilde Z_\ell: H^s_e(2\pi)\to H^s_e(2\pi)$
($s\in\mathbb{R}$) is given by
\begin{equation}\label{Z_op}
\tilde Z_\ell[\gamma](\theta')=\cos^\ell\theta'\;\tau(\cos\theta')\;
\gamma(\theta'),
\end{equation}
and where $A_4(\cos\theta,\cos\theta')\in C^{m+1}([0,2\pi]\times
[0,2\pi])$.  Now, in view of equation~\eqref{S0reg}, the first term on
the right-hand-side of equation~\eqref{Sfact_2} defines a bounded
operator from $H^s_e(2\pi)$ into $H^{s+1}_e(2\pi)$. On the other hand,
the derivatives of orders $k\leq (m+1)$ of the second term on the
right-hand-side of~(\ref{Sfact_2}), all reduce to integral operators
with bounded kernels, and thus map $L^2[0,2\pi]$ continuously into
$L^2[0,2\pi]$. It follows that the second term itself maps
continuously $H^0_e(2\pi)$ (and hence $H^{m}_e(2\pi)$) into
$H^{m+1}_e(2\pi)$, and the lemma follows for integer values $s=m$. The
extension for real values $s > 0$ follows directly by
interpolation~\cite[Theorem 8.13]{Kress}.
\end{proof}

The following lemma and its corollary provide a direct link, needed
for our proof of Lemma~\ref{Sinverse}, between the spaces
$H^s_e(2\pi)$ under consideration here and the original space
$\tilde{H}^{-\frac{1}{2}}(\Gamma)$ appearing in
equations~\eqref{Smapping}.

\begin{lemma}\label{LemmaStephanSpace0}
Let $s>0$, and assume $\tilde\varphi\in H^s_e(2\pi)$. Then the function
\begin{equation}
w(\xi)=\frac{1}{\pi}\int_0^\pi \tilde\varphi(\theta) e^{-i\xi\cos\theta}d\theta.
\end{equation} 
satisfies 
\begin{equation}
  \int_\mathbb{R}\frac{|w(\xi)|^2}{(1+|\xi|^2)^\frac{1}{2}}d\xi < \infty.
\end{equation}
\end{lemma}
\begin{proof}
  Using the $L^2[0,\pi]$-convergent cosine expansion
\begin{equation}
\tilde\varphi(\theta) = \sum_{n=0}^\infty a_n\cos\theta
\end{equation}
we obtain
\begin{equation}\label{fhat2}
  w(\xi)=\sum_{n=0}^\infty \frac{a_n}{\pi}\int_0^\pi \cos n\theta
 e^{-i\xi\cos\theta}d\theta.
 \end{equation}
Since 
\begin{equation}
  \int_0^\pi \cos n\theta e^{-i\xi\cos\theta}d\theta =
  \frac{1}{2}\int_{-\pi}^\pi e^{in\theta}e^{-i\xi\cos\theta}d\theta =
  \frac{1}{2}e^{\frac{in\pi}{2}}\int_{-\pi}^\pi e^{-in\theta}e^{-i\xi\sin\theta}d\theta = \pi i^n J_n(-\xi),
\end{equation}
(where, denoting by $J_n(\xi)$ the Bessel function of order $n$, the
last identity follows from~\cite[8.411 p. 902]{Gradshteyn}), we see
that equation~\eqref{fhat2} can be re-expressed in the form
\begin{equation}
  w(\xi ) = \sum\limits_{n=0}^\infty i^n a_n  J_n(-\xi) = \sum\limits_{n=0}^\infty \left( \sqrt{1 +n^{2s}}
    \; i^n a_n \right)\; \left( \frac{J_n(-\xi )}{\sqrt{1 +n^{2s}}}\right).
 \end{equation}
 In view of the Cauchy-Schwartz inequality we thus obtain
\begin{equation}
\left|w(\xi)\right|^2 \leq \left(\sum\limits_{n= 0}^\infty (1 +
n^{2s})|a_n|^2\right) \left(\sum\limits_{n=0}^\infty \frac{|J_n(\xi)|^2}{1+ n^{2s}}\right) \leq \left(
\sum\limits_{n=1}^\infty\frac{|J_n(\xi)|^2}{n^{2s}}
+|J_0(\xi)|^2\right)\|\tilde\varphi \|^2_s.
\end{equation}
Since $0\leq |\xi|/(1+|\xi|^2)^{1/2}\leq 1$, it follows that
\begin{equation}\begin{split}
\int_\mathbb{R}\frac{|w(\xi)|^2}{(1 + |\xi|^2)^\frac{1}{2}} d\xi\leq \left(
 \sum\limits_{n=1}^\infty
 \left(\frac{1}{n^{2s}}\int_\mathbb{R}\frac{|J_n(\xi)|^2}{(1 +
 |\xi|^2)^\frac{1}{2}}d\xi\right)+\int_\mathbb{R}\frac{|J_0(\xi)|^2}{(1+|\xi|^2)^\frac{1}{2}}d\xi\right)\|\tilde\varphi
 \|^2_s \\\leq \left(\sum\limits_{n=1}^\infty
 \left(\frac{1}{n^{2s}}\int_\mathbb{R}\frac{|J_n(\xi)|^2}{|\xi|}d\xi\right)+\int_\mathbb{R}\frac{|J_0(\xi)|^2}{(1+|\xi|^2)^\frac{1}{2}}d\xi\right)\|\tilde\varphi
 \|^2_s.
\end{split}
\end{equation}
Further, in view of \cite[6.574, eq 2.]{Gradshteyn}, the integral
involving $J_n$ can be computed exactly for $n\geq 1$:
\begin{equation}
\int_\mathbb{R} \frac{|J_n(\xi)|^2}{|\xi|}d\xi =\frac{1}{n}.
\end{equation}
It thus follows that
\begin{equation}
\int_\mathbb{R} \frac{| w(\xi) |^2}{(1 +|\xi |^2)^\frac{1}{2}}d\xi
\leq C_s \|\tilde\varphi \|^2_s < \infty
\end{equation} 
where
 \begin{equation}
 C_s= \sum\limits_{n=1}^\infty \frac{1}{n^{1+2s}} +
\int_\mathbb{R}\frac{|J_0(\xi)|^2}{(1+|\xi|^2)^\frac{1}{2}}d\xi.
\end{equation}
\end{proof}

\begin{corollary}\label{LemmaStephanSpace1} 
  Let $s>0$, $\tilde\varphi\in H^{s}_e(2\pi)$, $\varphi(t) =
  \tilde\varphi(\arccos(t))$, $\varphi:[-1,1]\to\mathbb{C}$,
  $\alpha(\bf{p}) = \varphi(\mathbf{r}^{-1}(\bf{p}))$ and $W(\bf{p}) =
  \omega(\mathbf{r}^{-1}(\bf{p}))$. Then, the function
  $F=\frac{\alpha}{W}$ is an element of
  $\tilde{H}^{-\frac{1}{2}}(\Gamma)$.
\end{corollary}
\begin{proof}
  It suffices to take show that $f = \varphi/\omega \in
  \tilde{H}^{-\frac{1}{2}}[-1,1]$ for the case $\Gamma =
  [-1,1]$. Extending $f$ by $0$ outside the interval $[-1,1]$, the
  Fourier transform of $f$ is given by
\begin{equation}
  \hat{f}(\xi)=\int_{-\infty}^\infty f(t)e^{-i\xi t }dt=\int_{-1}^1
  \frac{\varphi(t)e^{-i\xi t }}{\omega(t)} dt=\int_0^\pi
  \tilde\varphi(\theta) e^{-i\xi\cos\theta}d\theta,
\end{equation}
since $\omega(t)=\sqrt{1-t^2}$ in the present case.  The Corollary now
follows from Lemma~\ref{LemmaStephanSpace0}. \end{proof}

\begin{lemma}\label{Sinverse} For all $s > 0$ the operator
  $\tilde{S}: H^s_e(2\pi)\to H^{s+1}_e(2\pi)$ is invertible, and the
  inverse $\tilde{S}^{-1}: H^{s+1}_e(2\pi)\to H^s_e(2\pi)$ is a
  bounded operator.
\end{lemma}
\begin{proof}
  Let $s> 0$ be given.  From Lemma~\ref{lemma1} we know
  $\tilde{S_0}: H^s_e(2\pi)\to H^{s+1}_e(2\pi)$ is a continuously
  invertible operator. The same clearly holds for $\tilde{S}_0^\tau$
  as well, and we may write
\begin{equation}\label{s_tau}
\tilde{S} =\tilde{S_0^\tau}\left( I +
\left(\tilde{S_0^\tau}\right)^{-1}(\tilde{S}-\tilde{S_0^\tau})\right).
\end{equation}
It follows from Lemma~\ref{S_bounded} that the operator
$(\tilde{S}_0^\tau)^{-1}(\tilde{S}-\tilde{S_0^\tau})$ is bounded from
$H^s_e(2\pi)$ into $H^{s+1}_e(2\pi)$, and therefore, in view of the
Sobolev embedding theorem it defines a compact mapping from
$H^s_e(2\pi)$ into itself. Further, in view of
Corollary~\ref{LemmaStephanSpace1} and the injectivity of the
mapping~\eqref{Smapping} it follows that the operator
$\tilde{S}:H^s_e(2\pi)\to H^{s+1}_e(2\pi)$ is injective, and
therefore, so is
\begin{equation}\label{FredholmReduction}
\left(\tilde{S_0^\tau} \right)^{-1} \tilde{S} =
 I+\left(\tilde{S_0^\tau}\right)^{-1}(\tilde{S}-\tilde{S_0^\tau}):
 H^s_e(2\pi)\to H^{s}_e(2\pi).
 \end{equation}
A direct application of the Fredholm theory thus shows that the
operator~\eqref{FredholmReduction} is continuously invertible, and the lemma
follows.
\end{proof}

\subsection{Bicontinuity of the operator $\tilde{N}$}
To study the mapping properties of the operator $\tilde{N}$ we rely on
Lemma~\ref{Nfact_lemma} below where, as in~\cite{Monch}, the operator
$\tilde{N}$ is re-cast in terms of an expression which involves tangential
differential operators (cf.  also~\cite[Th. 2.23]{ColtonKress1} for the
corresponding result for closed surfaces). The needed relationships between
normal vectors, tangent vectors and parametrizations used are laid down in
the following definition.
\begin{definition}\label{tangent} 
  For a given (continuous) selection of the normal vector
  $\textbf{n}=\textbf{n}(\textbf{r})$ on $\Gamma$, the tangent vector
  $\textbf{t}(\textbf{r})$ is the unit vector that results from a $90^\circ$
  clockwise rotation of $\textbf{n}(\textbf{r})$. Throughout this paper it
  is further assumed that the parametrization $\textbf{r} = \textbf{r}(t)$
  of the curve $\Gamma$ has been selected in such a way that
  \begin{equation}
    \frac{d\textbf{r}}{dt}(t) = \left| \frac{d\textbf{r}}{dt} \right| \textbf{t}(\textbf{r}(t)).
  \end{equation}
\end{definition}
\begin{lemma}\label{Nfact_lemma}
  For $\varphi \in C^\infty(\Gamma)$, and for $t \in (-1,1)$, the
  quantity $N_\omega[\varphi](t)$ defined by
  equation~\eqref{NOmegaDef} can be expressed in the form
\begin{equation}\label{NomegaFact}
N_\omega[\varphi](t)=N_\omega^g[\varphi](t)+N_\omega^{pv}[\varphi](t)
\end{equation}
where \begin{equation}\label{NomegaG}
  N_\omega^g[\varphi](t)=k^2\int_{-1}^1
  G_k(\mathbf{r}(t),\mathbf{r}(t'))\;\varphi(t')\;
  \tau(t')\;\sqrt{1-t'^2}\;\mathbf{n}_t \cdot
  \mathbf{n}_{t'}\;dt',\end{equation} and where
\begin{equation}\label{NomegaPV}
  N_\omega^{pv}[\varphi](t)=\frac{1}{\tau(t)}\frac{d}{dt}\left(\int_{-1}^1
    G_k(\mathbf{r}(t),\mathbf{r}(t'))\;\frac{d}{dt'} \left(
      \varphi(t')\;\sqrt{1-t'^2}\right)dt'\right).
\end{equation}
\end{lemma}
\begin{proof}
  See Appendix~\ref{Nfact_appendix},
  cf.~\cite{Kress,ColtonKress1,Monch}.
\end{proof}


In order to continue with our treatment of the operator $\tilde{N}$ we
note that, using the changes of variables $t=\cos\theta$ and
$t'=\cos\theta'$ in equations~\eqref{NomegaG} and~\eqref{NomegaPV}
together with the notation~\eqref{phi_psi_theta}, for $\varphi\in
C^\infty(\Gamma)$ and for $\theta\in (0,\pi)$ we obtain
\begin{equation}\label{Nfact}
  \tilde{N}[\tilde\varphi]=\tilde{N}^{g}[\tilde\varphi]+\tilde{N}^{pv}[\tilde\varphi],
\end{equation}
where
\begin{equation}\label{Ng}
\tilde{N}^g[\tilde\varphi](\theta)=k^2\int_{0}^\pi
G_k(\mathbf{r}(\cos\theta),\mathbf{r}(\cos\theta'))
\;\tilde\varphi(\theta')\; \tau(\cos\theta')\;\sin^2\theta'\;
\mathbf{n}_\theta \cdot \mathbf{n}_{\theta'}\;d\theta',
\end{equation}
and where, taking into account equations~\eqref{D0Def} and~\eqref{T0Def},
\begin{equation}\label{Npv}
\tilde{N}^{pv}[\tilde\varphi](\theta) =
\frac{1}{\tau(\cos\theta)}\left(\tilde D_0 \tilde{S} \tilde
T_0^\tau\right)[\tilde\varphi](\theta),
\end{equation}
with 
\begin{equation}\label{t_0_tau}
\tilde T_0^\tau[\tilde\varphi](\theta)=\frac{1}{\tau(\cos\theta)}T_0[\tilde
\varphi](\theta).
\end{equation}

\begin{lemma}\label{N_pv_bounded}
Let $s \geq 0$. The operator $\tilde{N}^{pv}$ defines a bounded mapping from
$H^{s+1}_e(2\pi)$ to $H^s_e(2\pi)$. Further, the difference $(
\tilde{N}^{pv}-\tilde{N}_0^\tau )$ (see equation~\eqref{n0_tau_def}) defines a
bounded mapping from $H^{s+1}_e(2\pi)$ into $H^{s+1}_e(2\pi)$. \end{lemma}
\begin{proof}
Using~\eqref{N0Def},~\eqref{n0_tau_def} and~\eqref{Npv} we obtain
\begin{equation}\label{pv_0}
{\tilde N}^{pv}[\tilde \varphi] =\tilde{N}^\tau_0[\tilde \varphi] +
\frac{1}{\tau(\cos\theta)} \tilde D_0 (\tilde{S}-\tilde{S}_0^\tau) \tilde
T_0^\tau [\tilde \varphi].
\end{equation}
As shown in Theorem~\ref{j_0_tau_thm} the operator $\tilde{N}_0^\tau :
H^{s+1}_e(2\pi)\to H^s_e(2\pi)$ on the right-hand side of this equation is
bounded. To establish the continuity of the second term on the right-hand
side of equation~\eqref{pv_0} we first note that, in view of
equation~\eqref{T0Def}, the operator $\tilde T_0 :  H^{s+1}_e(2\pi)\to
H^s_e(2\pi)$ is bounded, and therefore, so is $\tilde T_0^\tau$.  Further,
as shown in Lemma~\ref{S_bounded}, the operator
$(\tilde{S}-\tilde{S}^\tau_0)$ maps continuously $H^s_e(2\pi)$ into
$H^{s+3}_e(2\pi)$ so that, to complete the proof, it suffices to show that
the operator $\tilde D_0$ maps continuously $H^{s+3}_e(2\pi)$ into
$H^{s+1}_e(2\pi)$. But, for $\tilde\psi\in H^{s+3}_e(2\pi)$ ($s>0$) we can
write
\[
\tilde D_0[\tilde \psi](\theta) = \frac{1}{\sin\theta}\int_0^\theta
\frac{d^2}{d\theta^2}\tilde\psi(u) du,
\]
and since the zero-th order term in the cosine expansion of
$\frac{d^2}{d\theta^2}\tilde\psi$ vanishes, in view of~\eqref{Cdef} we have
\[
\tilde D_0[\tilde \psi] = \tilde{C}\left[ \frac{d^2 \tilde
\psi}{d\theta^2}\right].
\]
It therefore follows from Lemma~\ref{Cprop} that the second term
in~\eqref{pv_0} is a continuous map from $H^{s+1}_e(2\pi)$ into
$H^s_e(2\pi)$, that is, $(\tilde{N}^{pv}-\tilde{N}_0^\tau )$, as claimed.
\end{proof}
\begin{corollary}\label{N_bounded}
  For all $s\geq 0$ the operator $\tilde{N}$ can be extended as a continuous
  linear map from $H^{s+1}_e(2\pi)$ to $H^s_e(2\pi)$. Further, the
  difference $\tilde{N}-\tilde{N}_0^\tau$ defines a continuous operator from
  $H^{s+1}_e(2\pi)$ to $H^{s+1}_e(2\pi)$.
\end{corollary}
\begin{proof}
  From equation~\eqref{Ng} we see that $\tilde{N}^g$ has the same
  mapping properties as $\tilde{S}$ (Lemma~\ref{S_bounded}), namely
\begin{equation}\label{ng_cont}
\tilde N^g: H^{s+1}_e(2\pi)\to H^{s+2}_e(2\pi) \quad \mbox{ is continuous}.
\end{equation}
In view of Lemma~\ref{N_pv_bounded} it therefore follows that the
right hand side of equation~\eqref{Nfact},
 \begin{equation}\label{Npvg}
   \tilde{N}^g+\tilde{N}^{pv} : H^{s+1}_e(2\pi)\to H^{s}_e(2\pi),
 \end{equation}
 is a bounded operator for all $s\geq 0$. Equation~\eqref{Nfact} was
 established for functions $\tilde\varphi$ of the
 form~\eqref{phi_psi_theta} with $\varphi\in C^\infty(\Gamma)$.  But
 the set of such functions $\tilde\varphi$ is dense in
 $H^{s+1}_e(2\pi)$ for all $s>0$---as can be seen by considering,
 e.g., that the Chebyshev polynomials span a dense set in
 $H^{s+1}[-1,1]$.  It follows that $\tilde N$ can be uniquely extended
 to a continuous operator from $H^{s+1}_e(2\pi)$ to $H^{s}_e(2\pi)$,
 as claimed. Finally, $\tilde N- \tilde{N}_0^\tau = \tilde{N}^g
 +(\tilde N^{pv}-\tilde{N}^\tau_0)$ is continuous from
 $H^{s+1}_e(2\pi)$ into $H^{s+1}_e(2\pi)$, in view of
 equation~\eqref{ng_cont} and Lemma~\ref{N_pv_bounded}.

\end{proof}

The following lemma establishes a link, needed for our proof of
Lemma~\ref{Ninverse}, between the domain of the unweighted
hypersingular operator $\mathbf{N}$ considered in~\cite{Stephan}
(equation~\eqref{Nmapping} above) and the corresponding possible
domains of the weighted operator $\tilde{N}$
(equation~\eqref{Ndomain_3}); cf. also
Corollary~\ref{LemmaStephanSpace1} where the corresponding result for
the domains of the operators $\mathbf{S}$ and $\tilde{S}$ is given.

\begin{lemma}\label{LemmaStephanSpace3}
  Let $\tilde\psi$ belong to $H^{s+1}_e(2\pi)$ for $s>0$,
  $\psi(t)=\tilde\psi( \arccos t)$, $\psi:\; [-1,1]\rightarrow
  \mathbb{C}$, $\beta(\bf{p})=\psi\left( \bf{r}^{-1}(\bf{p})\right)$,
  $W(\bf{p})=\omega\left( \textbf{r}^{-1}(\bf{p})\right)$. Then the function
  $G=W\beta$ is an element of $\tilde{H}^{\frac{1}{2}}(\Gamma)$.
\end{lemma}
\begin{proof}
  It suffices to show that $g=\omega\psi\in
  \tilde{H}^{\frac{1}{2}}[-1,1]$ for the case $\Gamma =
  [-1,1]$. Extending $g$ by $0$ outside the interval $[-1,1]$, the
  Fourier transform of $g$ is given by
\begin{equation}\label{gxi0}
  \hat{g}(\xi)=\int_{-1}^1 \psi(t)e^{-i\xi t} \omega(t) dt=\int_0^\pi \psi(\cos\theta)e^{-i\xi\cos\theta}\sin^2\theta d\theta,
\end{equation}
since $\omega(t)=\sqrt{1-t^2}$ in the present case. Integrating by
parts we obtain
\begin{equation}\label{gxi1}
\hat{g}(\xi)=\frac{1}{i\xi}\int_0^\pi \frac{\partial }{\partial \theta}\left\{\psi(\cos\theta)\sin\theta\right\} e^{-i\xi\cos\theta} d\theta.
\end{equation}
It is easy to check that $\frac{\partial}{\partial\theta}\{
  \psi(\cos\theta) \sin\theta\}=\frac{\partial}{\partial\theta}\{
  \tilde\psi(\theta) \sin\theta\}$ is an element of $H^s_e(2\pi)$
and, thus, in view of equation~\eqref{gxi1} together with
Lemma~\ref{LemmaStephanSpace0} we obtain
 \begin{equation}\label{gxi}
 \int_\mathbb{R} \frac{|\hat{g}(\xi)|^2\xi^2}{(1+\xi^2)^\frac{1}{2}}d\xi <\infty \end{equation}
 It thus follows that the second term on the right-hand side of
 the identity
 \begin{equation}
 \int_\mathbb{R} |\hat{g}(\xi)|^2(1 + |\xi|^2)^\frac{1}{2}d\xi = \int_\mathbb{R} \frac{|\hat{g}(\xi)|^2}{(1+\xi^2)^\frac{1}{2}}d\xi + \int_\mathbb{R} \frac{|\hat{g}(\xi)|^2\xi^2}{(1+\xi^2)^\frac{1}{2}}d\xi  
 \end{equation}
 is finite.  The first term is also finite, as can be seen by applying
 Lemma~\ref{LemmaStephanSpace0} directly to equation~\eqref{gxi0}. The
 function $g$ thus belongs to $\tilde{H}^\frac{1}{2}[-1,1]$, and the
 proof is complete.
\end{proof}

\begin{lemma}\label{Ninverse}
 For all $s > 0$ the operator $\tilde{N}: H^{s+1}_e(2\pi)\to H^{s}_e(2\pi)$
  is invertible, and the inverse $\tilde{N}^{-1}: H^{s}_e(2\pi)\to
  H^{s+1}_e(2\pi)$ is a bounded operator.
\end{lemma}
\begin{proof}
In view of Theorem~\ref{j_0_tau_thm}, the operator $\tilde{N}_0^\tau\;\,:\;
H^{s+1}_e(2\pi)\to H^s_e(2\pi)$ is bicontinuous, and we may thus write
\begin{equation}
\tilde{N}=\tilde{N^\tau_0}\left( I +\left(\tilde{N}^\tau_0\right)^{-1}(\tilde{N}-\tilde{N}^\tau_0)\right).
\end{equation}
Since, by Corollary~\ref{N_bounded}, the difference
$\tilde{N}-\tilde{N}^\tau_0$ defines a bounded mapping from
$H^{s+1}_e(2\pi)$ into $H^{s+1}_e(2\pi)$, it follows that the operator
$\left(\tilde{N}^\tau_0\right)^{-1}(\tilde{N}-\tilde{N}^\tau_0)$ is bounded
from $H^{s+1}_e(2\pi)$ into $H^{s+2}_e(2\pi)$ and, in view of the Sobolev
embedding theorem, it is also compact from $H^{s+1}_e(2\pi)$ into
$H^{s+1}_e(2\pi)$.  The Fredholm theory can thus be applied to the operator
\begin{equation}\label{id_p_comp}
I+\left(\tilde{N}^\tau_0\right)^{-1}(\tilde{N}-\tilde{N}^\tau_0).
\end{equation}
This operator is also injective, in view of Lemma~\ref{LemmaStephanSpace3}
and the bicontinuity of the map $ \mathbf{N}$ in equation~\eqref{Nmapping},
and it is therefore invertible. The Lemma then follows from the bicontinuity
of the operator of $\tilde{N}_0^\tau$.
\end{proof}

\section{Generalized Calder\'on Formula: Proof of
  Theorem~\ref{NStheorem}}\label{five}
Collecting results presented in previous sections we can now present a
proof of Theorem~\ref{NStheorem}.\\
 \begin{proof} The bicontinuity of the operators $\tilde{S}$,
   $\tilde{N}$ and $\tilde{N}\tilde{S}$ follow directly from Lemmas
   \ref{S_bounded},~\ref{Sinverse},~\ref{Ninverse} and
   Corollary~\ref{N_bounded}. To establish equation~\eqref{NSfact}, on the
   other hand, we write
\begin{equation}
\tilde{N}\tilde{S}=\tilde{N}_0^\tau\tilde{S}_0^\tau + \tilde K
=\tilde{J}_0^\tau + \tilde K
\end{equation}
where as shown in Theorem~\ref{j_0_tau_thm}, $\tilde{J}_0^\tau$ is
bicontinuous, and where
\begin{equation}
\tilde K=\tilde{N}(\tilde{S} - \tilde{S}_0^\tau) + (\tilde{N}- \tilde{N}_0^\tau)\tilde{S}_0^\tau.
\end{equation}
In view of Lemma~\ref{S_bounded}, Corollary~\ref{N_bounded} and
Theorem~\ref{j_0_tau_thm}, the operator $\tilde K$ maps $H^s_e(2\pi)$ into
$H^{s+1}_e(2\pi)$ and is therefore compact from $H^s_e(2\pi)$ into
$H^s_e(2\pi)$. The proof is now complete.\; \end{proof}

\section*{Acknowledgments}
The authors gratefully acknowledges support from the National Science
Foundation and the Air Force Office of Scientific Research.

\appendix
\section{Proof of Lemma~\ref{Nfact_lemma}\label{Nfact_appendix}}
\begin{proof}
  Assuming $\varphi\in C^\infty(\Gamma)$, we define the weighted
  double layer potential by 
\begin{equation}
\mathbf{D}_\omega[\alpha](\textbf{r})=\int_\Gamma \frac{\partial
  G(\textbf{r},\textbf{r}')}{\partial
  \textbf{n}_{\textbf{r}'}}\alpha(\textbf{r}')\omega(\textbf{r}')d\ell',\quad \mbox{\textbf{r} outside } \Gamma,
\end{equation}
which, following the related closed-surface calculation presented in
\cite[Theorem 2.23]{ColtonKress1}, we rewrite as 
\begin{equation}\label{Ndiv}
  \mathbf{D}_\omega[\alpha](\textbf{r})=-\mathrm{div}_{\textbf{r}} \textbf{E}[\alpha](\textbf{r})
\end{equation}
where
\begin{equation}\label{vecE}
  \textbf{E}[\alpha](\textbf{r}) = \int_\Gamma
  G_k(\textbf{r},\textbf{r}')\alpha(\textbf{r}')\omega(\textbf{r}')\textbf{n}(\textbf{r}')d\ell'.
\end{equation}
Since $\textbf{E}[\alpha] = \textbf{E}=(E_x,E_y)$ satisfies
$\Delta\textbf{E}+k^2\textbf{E} = 0$, the two dimensional gradient of
its divergence can be expressed in the form
\begin{equation}
  \mathrm{grad}\;\mathrm{div}\textbf{E}=-k^2\textbf{E}+ \left(\frac{\partial
  }{\partial y} \mathrm{curl}\;\textbf{E}, -\frac{\partial}{\partial
  x}\mathrm{curl}\;\textbf{E} \right),
\end{equation}
where the {\em scalar rotational} of a two-dimensional vector field
$\textbf{A} = (A_x,A_y)$ is defined by $\mathrm{curl}\;\textbf{A} =
(\frac{\partial A_y}{\partial x} -\frac{\partial A_x}{\partial y})$.
Since ${\rm curl}_{\textbf{r}}\left(\textbf{n}(\textbf{r}')
  G(\textbf{r},\textbf{r}') \right )$ equals
$-\textbf{t}(\textbf{r}')\cdot \nabla_{\textbf{r}'}
G(\textbf{r},\textbf{r}') $ (see definition~\ref{tangent}), we obtain
\begin{equation}\label{Fexp}
  {\rm curl}_{\textbf{r}}\textbf{E}[\alpha](\textbf{r})=-\int_{-1}^1 \frac{d G_k(\textbf{r},\textbf{r}(t'))}{dt'}\alpha(\mathbf{r}(t'))\omega(\mathbf{r}(t'))dt'.
\end{equation}
Therefore, taking the gradient of~\eqref{Ndiv}, letting $\varphi(t') =
\alpha(\mathbf{r}(t'))$, using~\eqref{canon_sq},
integrating~\eqref{Fexp} by parts, and noting that the boundary terms
vanish identically (since $\sqrt{1-t'^2} =0$ for $t' =\pm 1$), we see
that
\begin{equation}\label{gradDomega}
  \mbox{grad}\; \mathbf{D}_\omega[\alpha](\textbf{r})=k^2\int_\Gamma G(\textbf{r},\textbf{r}')\alpha(\textbf{r}')\textbf{n}(\textbf{r}')\omega(\mathbf{r}')d\ell'-\left(\frac{\partial A}{\partial y},-\frac{\partial A}{\partial x}\right),
\end{equation}
where
\begin{equation}\label{A}
  A(\textbf{r})=  \int_{-1}^1G(\textbf{r},\textbf{r}(t'))\frac{d}{dt'}
  \big(\varphi(t')\sqrt{1-t'^2}\big)dt'.
\end{equation}
In view of the continuity of the tangential derivatives of single layer
potentials across the integration surface (e.g.~\cite[Theorem
2.17]{ColtonKress1}), in the limit as $\textbf{r}\to\textbf{r}(t)\in \Gamma$
we obtain
\begin{equation}
  \left(\frac{\partial A(\textbf{r}(t))}{\partial y},-\frac{\partial A(\textbf{r}(t))}{\partial x}\right)\cdot\textbf{n}(\textbf{r}(t))=-\frac{1}{\tau(t)}\frac{dA(\textbf{r}(t))}{dt},
\end{equation}
and the decomposition~(\ref{NomegaFact}) results.
\end{proof}

\section{Asymptotic behavior of $\mathbf{N}\mathbf{S}[1]$\label{NSOne}}
In this section we demonstrate the poor quality of the composition
$\mathbf{N}\mathbf{S}$ of the unweighted hypersingular and single-layer
operators by means of an example: we consider the flat arc $[-1,1]$ at zero
frequency ($\mathbf{N}\mathbf{S}=\mathbf{N}_0\mathbf{S}_0$). In detail, in
Section~\ref{app_b1} we show that the image of $\mathbf{S}$ is not contained
in the domain of $\mathbf{N}$ (and, thus, the formulation
$\mathbf{N}\mathbf{S}$ cannot be placed in the functional
framework~\cite{Stephan,StephanWendland,StephanWendland2}), and in
Section~\ref{app_b2} we study the edge asymptotics of the function
$\mathbf{N}\mathbf{S}[1]$ which show, in particular, that the function $1$,
(which itself lies in $H^s[-1,1]$ for arbitrarily large values of $s$) is
mapped by the operator $\mathbf{N}\mathbf{S}$ into a function which does not
belong to the Sobolev space $H^{-\frac{1}{2}}[-1,1]$, and, thus, to any
space $H^s[-1,1]$ with $s\geq -1/2$.

We thus consider the {\em unweighted} single-layer and hypersingular
operators which, in the present flat-arc, zero-frequency case take
particularly simple forms. In view of~\eqref{Sdef}, the {\em parameter-space
form} of the unweighted single-layer operator (which is defined in a manner
analogous to that inherent in equation~\eqref{sOmegaDef} and related text)
is given by
\begin{equation}\label{S00}
S_0[\varphi](x)=-\frac{1}{2\pi}\int_{-1}^1\ln|x-s|\varphi(s)ds.
\end{equation}
With regards to the parameter-space form $N_0$ of the hypersingular
operator~\eqref{Ndef} we note, with reference to that equation, that in the
present zero-frequency flat-arc case we have $\mathbf{r}=(x,0)$,
$z\textbf{n}_{\mathbf{r}} = (0,z)$ and
$-d/d\textbf{n}_{\mathbf{r}'}=d/dz$. Since, additionally, the single layer
potential~\eqref{single_l} yields a solution of the Laplace equation in the
variables $(x,z)$, we have
\begin{equation}\label{N0zlimit}
4\pi N_0[\varphi](x)=-\lim_{z\rightarrow 0 }
\frac{d^2}{dx^2}\int_{-1}^1 \varphi(s)
\ln((x-s)^2+z^2)ds,
\end{equation}
or equivalently,
\begin{equation}\label{N0tangeant}
N_0[\varphi](x)=\frac{1}{4\pi}\lim_{z\rightarrow 0 }
\frac{d}{dx}\int_{-1}^1 \varphi(s) \frac{d}{ds} \ln((x-s)^2+z^2)ds.
\end{equation}
Note that, in view of the classical regularity theory for the Laplace
equation, letting $z$ tend to zero for $-1<x<1$ in equation
of~\eqref{N0zlimit} we also obtain, for smooth $\varphi$,
\begin{equation}\label{N00}
N_0[\varphi](x)=\frac{d^2}{dx^2}S_0[\varphi](x),\quad -1<x<1.
\end{equation}
\subsection{The operator $S_0$\label{app_b1}}
Integrating~\eqref{S00} by parts we obtain
\begin{equation}
-2\pi S_0[1](x)=\int_{-1}^1 \frac{d(s-x)}{ds}\ln|s-x|ds =
 (1-x)\ln(1-x) + (1+x)\ln(1+x) -2,
\end{equation}
and therefore
\begin{equation}\label{S0base}
S_0[1](x)=\frac{1}{2\pi}\big( 2 - (1-x)\ln(1-x) -(1+x)\ln(1+x)\big).
\end{equation}
Incidentally, this expression shows that the unweighted
single-layer operator does not map $C^\infty$ functions into $C^\infty$
functions up to the edge; a more general version of this result is given
in~\cite[p. 182]{Tricomi}.

The following two lemmas provide details on certain mapping properties of the operator $S_0$ .

\begin{lemma}\label{S0reg1} 
The image $S_0[1]$ of the constant function $1$ by the operator~\eqref{S00}
is an element of $H^\frac{1}{2}[-1,1]$.
\end{lemma} 
\begin{proof}
Let $\Gamma_1$ be a closed, smooth curve which includes the segment
$[-1,1]$. Clearly, the function
\begin{equation}\label{f1}
f_1(s)=\left\{ \begin{array}{ll}1,& s \in [-1,1]\\ 0,& s\in
  \Gamma_1\backslash [-1,1]\end{array}\right.
\end{equation}
belongs to $L^2(\Gamma_1)$ and therefore to $H^{-\frac{1}{2}}(\Gamma_1)$, so
that, according to Definition~\ref{htil_def}, the constant 1 is in the space
$\tilde{H}^{-\frac{1}{2}}[-1,1]$. In view of equation~\eqref{Smapping}, it
follows that $S_0[1]\in H^{\frac{1}{2}}[-1,1]$.
\end{proof}
\begin{lemma}\label{S0reg_ex} 
The image $S_0[1]$ of the constant function $1$ by the operator~\eqref{S00}
is not an element of $\tilde H^\frac{1}{2}[-1,1]$.
 
\end{lemma} 
\begin{proof} 
In view of~\eqref{S0base} and the fact that $S_0[1](x)$ is an even function
of $x$, integration by parts yields
\begin{equation}\label{S0IBP}
\int_{-1}^1 e^{-i\xi x}S_0[1](x)dx = \frac{2\sin \xi}{\xi}S_0[1](1)+
  \frac{1}{2\pi i\xi}\int_{-1}^1 e^{-i\xi x
  }\ln\frac{1-x}{1+x}dx.\end{equation} Taking into account the identities
  \cite[eq. 4.381, p 577]{Gradshteyn}
\begin{equation}\label{sinlogint}
\left\{\begin{array}{ll} \int_0^1\ln x \cos \xi x\, dx &=
-\frac{1}{\xi}\left[\mathrm{si}(\xi) + \frac{\pi}{2} \right], \\ \int_0^1\ln
x \sin \xi x\, dx &= -\frac{1}{\xi}\left[\textbf{C} +\ln \xi -
  \mathrm{ci}(\xi) \right],\end{array} \right. \quad \xi>0
\end{equation}
where $\mathbf{C}$ is the Euler constant, and where $\mathrm{si}(\xi)$ and
$\mathrm{ci}(\xi)$ are the sine and cosine integrals respectively, (both of
which are bounded functions of $\xi$ as $|\xi|$ tends to infinity), it is
easily verified that the second term in~\eqref{S0IBP} behaves asymptotically
as $\frac{\ln(\xi)}{\xi^2}$ as $\xi$ tends to infinity. Clearly, the first
term of~\eqref{S0IBP} decays as $O(\frac{1}{\xi})$, and therefore
\begin{equation}\label{FT_S0}
\left|\int_{-1}^1 e^{-i\xi x}\left(S_0[1](x)\right) dx \right|^2 =
 O\left(\frac{1}{\xi^2}\right), \quad \xi \rightarrow \infty.
\end{equation}

Equation~\eqref{FT_S0} tells us that the function $\varphi:\mathbb{R}\to
\mathbb{R}$ which equals $S_0[1](x) $ for $x$ in the interval
$[-1,1]$ and equals zero in the complement of this interval, does not belong
to $H^\frac{1}{2}\left(\mathbb{R} \right)$, and, thus, $S_0[1]\not\in\tilde
H^\frac{1}{2}[-1,1]$, as claimed.
\end{proof}
\begin{remark}\label{S0remark}
Lemmas ~\ref{S0reg1} and ~\ref{S0reg_ex} demonstrate that, as pointed out in
Section~\ref{intro}, the formulation $\mathbf{N}\mathbf{S}$ of the
open-curve boundary-value problems under consideration cannot be placed in
the functional framework put forth
in~\cite{Stephan,StephanWendland,StephanWendland2} and embodied by
equations~\eqref{Smapping},~\eqref{Nmapping} and definition~\ref{htil_def}:
the image of the operator $\mathbf{S}$ is not contained in the domain of
definition of the operator $\mathbf{N}$; see equations~\eqref{Smapping}
and~\eqref{Nmapping}.
\end{remark}

\subsection{The combination $N_0S_0$\label{app_b2}}
While, as pointed out in the previous section, $S_0[1]$ does not belong to
the domain of definition of $N_0$ (as set up by the
formulation~\eqref{Smapping},~\eqref{Nmapping}), the quantity $N_0S_0[1](x)$
can be evaluated pointwise for $|x|<1$, and it is instructive to study its
asymptotics as $x\to \pm 1$.
\begin{lemma}\label{N0L2}
$N_0S_0[1]$ can be expressed in the form
\begin{equation}\label{N0S01}
N_0S_0[1](x)=\frac{\ln2-1}{\pi^2(1-x^2)} + \mathcal{L}(x), 
\end{equation}
where $\mathcal{L}\in L^2[-1,1]$.
\end{lemma}
\begin{proof} 
In view of~\eqref{S0base} we have
\begin{equation}\label{N0S0_start}
N_0S_0[1](x)=\frac{1}{\pi}N_0[1](x) - \frac{1}{2\pi}N_0[g](x),
\end{equation}
where 
\begin{equation} 
g(x)=(1-x)\ln(1-x) + (1+x)\ln(1+x).
\end{equation} 
For the first term on the right-hand side of this equation we obtain
from~\eqref{N00} and~\eqref{S0base}
\begin{equation}\label{N0_1}
N_0[1](x)= -\frac{1}{\pi(1-x^2)}\, .
\end{equation}
To evaluate the second term $N_0[g]$ in equation~\eqref{N0S0_start}, in
turn, we first integrate by parts equation~\eqref{N0tangeant} and take limit
as $z\to 0$ and thus obtain
\begin{equation}
\begin{split}
N_0[g](x)=\frac{1}{2\pi}\left(\frac{d}{dx}\left(\big[\ln|x-s|g(s)\big]_{-1}^1\right) - \frac{d}{dx}\int_{-1}^1
\ln|x-s| \frac{d}{ds}g(s)ds\right)\\
=\frac{\ln 2}{\pi}\frac{d}{dx}\left(\ln\left(\frac{1-x}{1+x}\right) \right)
+ \frac{1}{2\pi}\frac{d}{dx}\int_{-1}^1
\ln|x-s| \ln\left(\frac{1-s}{1+s}\right)ds,
\end{split}
\end{equation}
or
\begin{equation}\label{N0pv}
N_0[g](x)=\frac{-2\ln2}{\pi(1-x^2)}-\frac{1}{2\pi}p.v.\int_{-1}^1\ln\left(\frac{1-s}{1+s}\right)\frac{1}{s-x}ds.
\end{equation}
Clearly, to complete the proof it suffices to establish that the functions
\begin{equation}
\mathcal{L}^+(x)=p.v.\int_{-1}^1\frac{\ln(1-s)}{s-x}ds\quad\mbox{and}\quad
\mathcal{L}^-(x)=p.v.\int_{-1}^1\frac{\ln(1+s)}{s-x}ds
\end{equation}
are elements of $L^2[-1,1]$.

Let us consider the function $\mathcal{L}^+$ for $x\geq 0$
first. Re-expressing $\mathcal{L}^+(x)$ as the sum of the integrals over the
interval $[x-(1-x), x + (1-x)] = [2x-1,1]$ (which is symmetric respect to
$x$ plus the integral over $[-1,2x-1]$ and using a simple change of
variables we obtain
\begin{equation}\label{L0}
\mathcal{L}^+(x)= \int_0^{1-x}\frac{\ln(1-x-u)-\ln(1-x+u)}{u}du +
\int_{1-x}^{1+x}\frac{\ln(1-x+u)}{u}du.
\end{equation}
Letting $z = 1-x$ and $v = \frac{u}{z}$, we see that the first integral
in~\eqref{L0} is a constant function of $x$:
\begin{equation}\label{FirstInt}
 \int_0^{z}\frac{\ln(z-u)-\ln(z+u)}{u}du =
 \int_0^1\frac{\ln(1-v)-\ln(1+v)}{v}dv = \mbox{const}.
\end{equation}
For the second integral in~\eqref{L0}, on the other hand, we write
\begin{equation}\label{SecondInt}
\begin{split}
\int_{1-x}^{1+x}\frac{\ln(1+u-x)}{u}du
=\int_{1-x}^{1+x}\frac{\ln(1+\frac{u}{1-x})}{u}du +
\ln(1-x)\int_{1-x}^{1+x}\frac{du}{u}\\
=\int_1^{\frac{1+x}{1-x}}\frac{\ln(1+v)}{v}dv +
\ln(1-x)\ln\left(\frac{1+x}{1-x}\right)\\
=\int_1^{\frac{1+x}{1-x}}\frac{\ln(1+v)}{1+v}dv
+\int_1^{\frac{1+x}{1-x}}\frac{\ln(1+v)}{v(1+v)}dv +
\ln(1-x)\ln\left(\frac{1+x}{1-x}\right)\\
=\frac{1}{2}\left(\ln^2\left(\frac{2}{1-x}\right)-\ln^2 2\right)
+\int_1^{\frac{1+x}{1-x}}\frac{\ln(1+v)}{v(1+v)}dv +
\ln(1-x)\ln\left(\frac{1+x}{1-x}\right).
\end{split}
\end{equation}
Since the second term on the last line of equation~\eqref{SecondInt} is
bounded for $0\leq x<1$, it follows that, in this interval, the function
$\mathcal{L}^+(x)$ equals a bounded function plus a sum of logarithmic terms
and is thus an element of $L^2[0,1]$. Using a similar calculation it is
easily shown that $\mathcal{L}^+(x)$ is bounded for $-1\leq x <0$, and it
thus follows that $\mathcal{L}^+\in L^2[-1,1]$, as desired. Analogously, we
have $\mathcal{L}^-\in L^2[-1,1]$, and the lemma follows.
\end{proof}
\begin{corollary}
Let $\Gamma = [-1,1]$. Then $\mathbf{N}\mathbf{S}[1]$ does not belong to the
codomain $H^{-\frac{1}{2}}[-1,1]$ of the operator $\mathbf{N}$ in
equation~\eqref{Nmapping}.
\end{corollary}
\begin{proof}
In view of Lemma~\ref{N0L2} it suffices to show that the function
$h(x)=\frac{1}{1-x^2}$ does not belong to $H^{-\frac{1}{2}}[-1,1]$, or,
equivalently, that the primitive $k(x)=-\frac{1}{2}\ln\frac{1-x}{1+x}$ of $h$ does not
belong to $H^\frac{1}{2}[-1,1]$.  Clearly, to establish that $k\not\in
H^\frac{1}{2}[-1,1]$ it suffices to show that the function
$\ell(x)=p(x)\ln(x)$ is not an element of $H^\frac{1}{2}[0,\infty[$, where
$p$ is a smooth auxiliary function defined for $x\geq 0$ which equals 1 in
the interval $[0,1]$ and which vanishes outside the interval $[0,2]$. 

To do this we appeal to the criterion~\cite[p. 54]{LionsMagenes}
\[
\ell\in H^\frac{1}{2}(0,\infty) \iff \ell\in L^2(0,\infty) \quad \mbox{and}
\int_0^\infty t^{-2}dt\int_0^\infty \left|\ell(x+t) - \ell(x) \right|^2dx<
\infty.
\]
To complete the proof of the lemma it thus suffices to show that the
integral
\[
I = \int_0^\infty t^{-2}dt\int_0^1 \left|\ln(x+t) - \ln(x) \right|^2dx
\]
is infinite. But, using the change of variables $u=\frac{t}{x}$ we obtain
\begin{equation}
I=\int_0^{\infty}\frac{1}{t}dt\left(\int_t^\infty
\frac{\left|\ln(1+u)\right|^2}{u^2}du\right)=\infty,
\end{equation}
and the lemma follows.
\end{proof}

%

\bibliographystyle{plain}
\bibliography{./Screens}

\end{document}